\documentclass[reqno,11 pt]{amsart}

\usepackage{amscd}
\usepackage{a4wide}
\usepackage[english]{babel}
\usepackage[utf8]{inputenc}
\usepackage{amsmath,amsthm,amsfonts}
\usepackage{braket}
\usepackage{hyperref}
\usepackage{relsize}
\usepackage{tikz}

\theoremstyle{definition}
\newtheorem{defi}{Definition}[section]
\newtheorem{lem}[defi]{Lemma}
\newtheorem{prop}[defi]{Proposition}
\newtheorem{thm}[defi]{Theorem}
\newtheorem{cor}[defi]{Corollary}
\newtheorem{rem}[defi]{Remark}

\renewcommand{\Re}{\operatorname{Re}}
\renewcommand{\Im}{\operatorname{Im}}
\newcommand{\Arg}{\operatorname{Arg}}
\newcommand{\ran}{\operatorname{ran}}
\newcommand{\dom}{\operatorname{dom}}

\title[]{The harmonic $H^\infty$-functional calculus \\ based on the $S$-spectrum}

\author[A. de Martino]{Antonino De Martino}
\address{(AM) Politecnico di Milano, Dipartimento di Matematica, Via E. Bonardi 9, 20133 Milano, Italy}
\email{antonino.demartino@polimi.it}

\author[S. Pinton]{Stefano Pinton}
\address{(SP) Politecnico di Milano, Dipartimento di Matematica, Via E. Bonardi 9, 20133 Milano, Italy}
\email{stefano.pinton@polimi.it}

\author[P. Schlosser]{Peter Schlosser}
\address{(PS) Politecnico di Milano, Dipartimento di Matematica, Via E. Bonardi 9, 20133 Milano, Italy}
\email{pschlosser@math.tugraz.at}

\begin{document}

\begin{abstract}
The aim of this paper is to introduce the $H^\infty$-functional calculus for harmonic functions over the quaternions. More precisely, we give meaning to $Df(T)$ for unbounded sectorial operators $T$ and polynomially growing functions of the form $Df$, where $f$ is
a slice hyperholomorphic function and $D=\partial_{q_0}+e_1\partial_{q_1}+e_2\partial_{q_2}+e_3\partial_{q_3}$ is the Cauchy-Fueter operator. The harmonic functional calculus can be viewed as a modification of the well known $S$-functional calculus $f(T)$, with a different resolvent operator. The harmonic $H^\infty$-functional calculus is defined in two steps: First, for functions with a certain decay property, one can make sense of the bounded operator $Df(T)$ directly via a Cauchy-type formula. In a second step, a regularization procedure is used to extend the functional calculus to polynomially growing functions and consequently unbounded operators $Df(T)$.
 The harmonic functional calculus is an important functional calculus of the quaternionic fine structures on the $S$-spectrum, which
  arise also in the Clifford setting and they
 encompass a variety of function spaces and the corresponding functional calculi.
 These function spaces emerge through  all possible  factorizations  of the second map of the Fueter-Sce extension theorem.
 This field represents an emerging and expanding research area that serves as a bridge
 connecting operator theory, harmonic analysis, and hypercomplex analysis.

 \end{abstract}

\maketitle

AMS Classification: 47A10, 47A60

Keywords: H-infinity functional calculus, harmonic functional calculus, Cauchy-Fueter operator, $S$-spectrum, fine structures.

\medskip

\textbf{Acknowledgements:} The research of P. Schlosser was funded by the Austrian Science Fund (FWF) under Grant No. J 4685-N and by the European Union -- NextGenerationEU.

\section{Introduction}

The complex Riesz-Dunford functional calculus \cite{RD} is based on the Cauchy integral formula and gives meaning to holomorphic functions of operators via the integral
\begin{equation}\label{Eq_Riesz_Dunford_functional_calculus}
f(A):=\frac{1}{2\pi i}\int_{\partial \Omega}(z\mathcal{I}-A)^{-1}f(z)dz,
\end{equation}
where $\Omega\subset\mathbb{C}$ is a suitable open set containing the spectrum of the complex bounded linear operator $A$.
Extending the holomorphic functional calculus to quaternionic operators presented a formidable challenge because of the various potential concepts of quaternionic holomorphicity.
Indeed, in hypercomplex analysis, unlike in complex analysis, it is important to recognize the existence of multiple notions of analyticity, including slice hyperholomorphic functions and Cauchy-Fueter regular functions.
Moreover, in the slice hyperholomorphic setting, due to the lack of commutativity, there exist two different hyperholomorphic Cauchy kernels
\begin{equation}\label{Eq_Cauchy_kernel}
S_L^{-1}(s,q):=(s-\overline{q})(s^2-2q_0s+|q|^2)^{-1}\quad\text{and}\quad S_R^{-1}(s,q):=(s^2-2q_0s+|q|^2)^{-1}(s-\overline{q}).
\end{equation}
For $f$ in the space $\mathcal{SH}_L(U)$ of left slice hyperholomorphic functions, the left Cauchy kernel $S_L^{-1}(s,q)$ gives raise to the Cauchy integral formula
\begin{equation}\label{Eq_Cauchy_formula}
f(q)=\frac{1}{2\pi}\int_{\partial(U\cap\mathbb{C}_J)}S_L^{-1}(s,q)ds_Jf(s),
\end{equation}
while for right slice hyperholomorphic functions a similar formula holds true using the right Cauchy kernel $S_R^{-1}(s,q)$, see for example \cite[Theorem 2.1.32]{CGK}. For the interpretation of the integral and the notion of slice hyperholomorphicity we refer to Definition~\ref{defi_Path_integral} and Definition~\ref{defi_Slice_hyperholomorphic_functions}. Motivated by the integral \eqref{Eq_Cauchy_formula}, the quaternionic extension of \eqref{Eq_Riesz_Dunford_functional_calculus}, called $S$-functional calculus for a quaternionic linear operator $T$, is defined by
\begin{equation}\label{Eq_S_functional_calculus_formal}
f(T):=\frac{1}{2\pi}\int_{\partial(U\cap\mathbb{C}_J)}S_L^{-1}(s,T)ds_Jf(s),
\end{equation}
where the left $S$-resolvent operator $S_L^{-1}(s,T)$ is associated with \eqref{Eq_Cauchy_kernel} and properly defined in \eqref{Eq_S_resolvent}. In this paper we establish a functional calculus similar to \eqref{Eq_S_functional_calculus_formal}, for functions in the space
\begin{equation*}
\mathcal{AH}_L(U)=\Set{Df | f\in\mathcal{SH}_L(U)},
\end{equation*}
where $D$ is the Cauchy-Fueter operator
\begin{equation}\label{Eq_Cauchy_Fueter_operator}
D=\frac{\partial}{\partial q_0}+e_1\frac{\partial}{\partial q_1}+e_2\frac{\partial}{\partial q_2}+e_3\frac{\partial}{\partial q_3}.
\end{equation}
Due to the Fueter mapping theorem \cite{Fueter}, the space $\mathcal{AH}_L(U)$ is exactly the one of axially harmonic functions. Applying the operator $D$ to the formula \eqref{Eq_Cauchy_formula}, carrying the derivatives inside the integral and using the explicit formula
\begin{equation*}
DS_L^{-1}(s,q)=-2(s^2-2q_0s+|q|^2)^{-1}=:-2Q_{c,s}^{-1}(q)
\end{equation*}
of the Cauchy-kernel, we obtain the Cauchy-type formula for axially harmonic functions
\begin{equation*}
Df(q)=\frac{-1}{\pi}\int_{\partial(U\cap\mathbb{C}_J)}Q_{c,s}^{-1}(q)ds_Jf(s).
\end{equation*}
Replacing the quaternion $q$ by a suitable quaternionic operator $T$ the above formula motivates the definition of the harmonic functional calculus
\begin{equation}\label{Eq_Q_functional_calculus_formal}
Df(T)=\frac{-1}{\pi}\int_{\partial(U\cap\mathbb{C}_J)}Q_{c,s}^{-1}(T)ds_Jf(s).
\end{equation}
For a rigorous definition of the operator $Q_{c,s}^{-1}(T)$ see \eqref{Eq_Commutative_Q_resolvent} and \eqref{Eq_Resolvent_set}. For bounded operators $T$, this functional calculus is already well established in \cite{CDPS1,Polyf2}. In this paper we consider the class of unbounded sectorial operators (also called operators of type $\omega$). The set $U$ mentioned in the boundary integration within equation \eqref{Eq_Q_functional_calculus_formal} must encompass the $S$-spectrum of $T$. So for this class of operators  the $S$-spectrum is unbounded and so we have to require additional decay properties on the function $f$ to ensure the convergence of the integral.

\medskip

Nonetheless, by employing a regularization procedure known as the $H^\infty$-functional calculus, originally introduced by A. McIntosh in the complex context \cite{McI1}, we can subsequently  extend the harmonic functional calculus to include functions $f$ that exhibit polynomial growth. To explore the $H^\infty$-functional calculus in the complex context, one can refer to the following sources: \cite{Haase,HYTONBOOK2,HYTONBOOK1,W}. Additionally, for an extension of the $H^\infty$-functional calculus in the monogenic Clifford setting and its applications, one can consider the books \cite{JBOOK, TAOBOOK}.

\medskip

We now intend to offer a comprehensive perspective on our theory by elucidating its initial motivations, detailing recent developments, and introducing a new branch known as the fine structures on the $S$-spectrum. Within this new branch, we encounter the development of the harmonic functional calculus, whose $H^\infty$-version has been elaborated in this work.

\medskip

The study of quaternionic operators is motivated by their relevance to the formulation of quantum mechanics. The foundational work by Birkhoff and von Neumann \cite{BF} showed that there are essentially two ways to formulate quantum mechanics, one using complex numbers and the other using quaternions. This result is significant since it establishes the importance of quaternionic operators as a fundamental part of quantum mechanics. The notion of $S$-spectrum was discovered in 2006, only using methods in hypercomplex analysis even though its existence was suggested by quaternionic quantum mechanics. For more details see the introduction of the book \cite{CGK}. The notion of $S$-spectrum extends also to Clifford operators, see \cite{ColomboSabadiniStruppa2011}, and recently it has been shown that the quaternionic and the Clifford settings are just particular cases of a more general framework in which the spectral theory on the $S$-spectrum can be developed, see \cite{ADVCGKS,PAMSCKPS} and the references therein. Using the notion of $S$-spectrum it was also possible to prove the quaternionic version of the spectral theorem. Precisely, unitary operators, using Herglotz's functions, were considered in \cite{ACKS16}, while quaternionic normal operators are studied in \cite{ACK}. More recently, the spectral theorem based on the $S$-spectrum was extended also to Clifford operators, see \cite{ColKim}.

\medskip

The development of the spectral theory on the $S$-spectrum has opened up several research directions in hypercomplex analysis and operator theory, without claiming completeness, we mention the quaternionic perturbation theory and invariant subspaces \cite{CereColKaSab} and new classes of fractional diffusion problems \cite{ColomboDenizPinton2020,ColomboDenizPinton2021,CGdiffusion2018,FJBOOK,ColomboPelosoPinton2019} that are based on the $H^\infty$-version of the $S$-functional calculus \cite{ACQS2016,CGdiffusion2018}. Finally, we mention that the spectral theory on the $S$-spectrum is systematically organized in the books \cite{FJBOOK,CGK,ColomboSabadiniStruppa2011}.

\medskip

In recent times a new branch of the spectral theory on the $S$-spectrum  has been developed, that is called fine structures on the $S$-spectrum. It turns out that the space $\mathcal{AH}_L(U)$ of axially harmonic functions is only one of the spaces arising from the factorization of the operators in the Fueter-Sce theorem \cite{Fueter,TaoQian1,Sce}, which connects the class of slice hyperholomorphic functions with the class of axially monogenic functions $\mathcal{AM}_L(U)$ via the powers $\Delta^{\frac{n-1}{2}}$ of the Laplace operator in dimension $n+1$, where $n$ is the number of imaginary units in the Clifford algebra $\mathbb{R}_n$. Their corresponding integral representations in turn give raise to various functional calculi, see \cite{CDPS1,Fivedim,Polyf1,Polyf2}. Note, that for odd $n$ the operator $\Delta^{\frac{n-1}{2}}$ is a pointwise differential operator, see \cite{ColSabStrupSce,Sce}, while for even values of $n$ we are dealing with fractional powers of the Laplace operator, see \cite{TaoQian1}. For more information on the work of M.~Sce see the translation of his work with in \cite{ColSabStrupSce}, and for a different description of the Fueter-Sce theorem \cite{DDG, DDG1}.

\medskip

In the special case of the quaternions $\mathbb{H}$ (which are classically identified with $\mathbb{R}_2$), the Fueter-Sce theorem, called Fueter mapping theorem in this case, holds true with $n=3$. This means, the connection between slice hyperholomorphic and axially monogenic functions is via the four dimensional Laplace operator
\begin{equation*}
\Delta=\frac{\partial^2}{\partial q_0^2}+\frac{\partial^2}{\partial q_1^2}+\frac{\partial^2}{\partial q_2^2}+\frac{\partial^2}{\partial q_3^2}.
\end{equation*}
This operator can now be factorized using the Cauchy-Fueter operator $D$ from \eqref{Eq_Cauchy_Fueter_operator} and its conjugate $\overline{D}$ in the two different ways
\begin{equation}\label{Eq_Delta_factorization}
\Delta=D\overline{D}=\overline{D}D.
\end{equation}
The space $\mathcal{AH}_L(U)$, for which we motivated the harmonic functional calculus in \eqref{Eq_Q_functional_calculus_formal}, now arises from the first of the above factorizations, namely
\begin{equation*}
\begin{CD}
\mathcal{SH}_L(U) @>D>> \mathcal{AH}_L(U) @>\overline{D}>> \mathcal{AM}_L(U).
\end{CD}
\end{equation*}
At this point we also want to mention that the polyanalytic functional calculus of the space
\begin{equation*}
\mathcal{AP}_{2,L}(U)=\Set{\overline{D}f | f\in\mathcal{SH}_L(U)},
\end{equation*}
arising from the second factorization in \eqref{Eq_Delta_factorization}, is constructed with similar strategy in \cite{CDP23,Polyf1} and the $\mathcal{F}$-functional calculus for the space $\mathcal{AM}_L(U)$ was established long ago in \cite{CSS}, while more recent investigations can be found in \cite{CDS,CDS1,CG,CS,CSF}.

\medskip

\textit{Plan of the paper}. The aim of Section~\ref{sec_Preliminary_results_on_quaternionic_function_theory} is to collect the basic facts about quaternionic function theory and in particular to fix the notations used in the paper. In Section~\ref{sec_Harmonic_functional_calculus_for_decaying_functions} we investigate the integral \eqref{Eq_Q_functional_calculus_formal} for operators $T$ of type $\omega$ and slice hyperholomorphic functions $f$ which have some order of decay at $0$ and at $\infty$. Beside the welldefinedness of this functional calculus stated in Theorem~\ref{thm_Q_functional_calculus_decaying}, we also prove its basic properties as the commutation with other operators in Proposition~\ref{prop_Commutation_with_Df} and Corollary~\ref{cor_Commutation_g_with_Df}, the product rule in Theorem~\ref{thm_Product_rule_decaying} and the equivalence to the rational functional calculus in Proposition~\ref{prop_Rational_equivalence_decaying}. In Section~\ref{sec_Harmonic_functional_calculus_for_growing_functions} we extend the theory of Section~\ref{sec_Harmonic_functional_calculus_for_decaying_functions} to functions which are allowed to grow polynomially at $0$ and at $\infty$, using certain regularizer functions. Again, the commutation relations in Proposition~\ref{prop_Commutation_with_Df_growing} and Corollary~\ref{cor_Commutation_g_with_Df_growing}, the product rule in Theorem~\ref{thm_Product_rule_growing} and the equivalence to the rational functional calculus in Proposition~\ref{prop_Rational_equivalence_growing} are also proven for this extended theory.

\section{Preliminary results on quaternionic function theory}\label{sec_Preliminary_results_on_quaternionic_function_theory}

The setting in which we will work in this paper is the one of the \textit{quaternions}
\begin{equation*}
\mathbb{H}:=\Set{s_0+s_1e_1+s_2e_2+s_3e_3 | s_0,s_1,s_2,s_3\in\mathbb{R}}
\end{equation*}
with the three imaginary units $e_1,e_2,e_3$ satisfying the relations
\begin{equation*}
e_1^2=e_2^2=e_3^2=-1\qquad\text{and}\qquad\begin{array}{l} e_1e_2=-e_2e_1=e_3, \\ e_2e_3=-e_3e_2=e_1, \\ e_3e_1=-e_1e_3=e_2. \end{array}
\end{equation*}
For every quaternion $s\in\mathbb{H}$ we set
\begin{align*}
\Re(s)&:=s_0, && (\textit{real part}) \\
\Im(s)&:=s_1e_1+s_2e_2+s_3e_3, && (\textit{imaginary part}) \\
\overline{s}&:=s_0-s_1e_1-s_2e_2-s_3e_3, && (\textit{conjugate}) \\
|s|&:=\sqrt{s_0^2+s_1^2+s_2^2+s_3^2}. && (\textit{modulus})
\end{align*}
The unit sphere of purely imaginary quaternions is defined as
\begin{equation*}
\mathbb{S}:=\Set{s\in\mathbb{H} | s_0=0\text{ and }|s|=1},
\end{equation*}
and for every $J\in\mathbb{S}$ consider the complex plane
\begin{equation*}
\mathbb{C}_J:=\Set{u+Jv | u,v\in\mathbb{R}},
\end{equation*}
which is an isomorphic copy of the complex numbers, since every $J\in\mathbb{S}$ satisfies $J^2=-1$. Moreover, for every quaternion $s\in\mathbb{H}$ we consider the corresponding $2$-sphere
\begin{equation*}
[s]:=\Set{\Re(s)+J\vert\Im(s)\vert | J\in\mathbb{S}}.
\end{equation*}
Next, we introduce the so called \textit{slice hyperholomorphic functions}, which are a quaternionic analog to the complex holomorphic functions. The sets on which those functions are defined and suitable for operator theory
are the following \textit{axially symmetric sets}.

\begin{defi}[Axially symmetric sets]
A subset $U\subseteq\mathbb{H}$ is called \textit{axially symmetric}, if $[s]\subseteq U$ for every $s\in U$.
\end{defi}

\begin{defi}[Slice hyperholomorphic functions]\label{defi_Slice_hyperholomorphic_functions}
Let $U\subseteq\mathbb{H}$ be axially symmetric and open. A function $f:U\rightarrow\mathbb{H}$ is called \textit{left} (resp. \textit{right}) \textit{slice hyperholomorphic}, if there exists two continuously differentiable functions $\alpha,\beta:\mathcal{U}:=\Set{(u,v)\in\mathbb{R}^2 | u+\mathbb{S}v\subseteq U}\rightarrow\mathbb{H}$, with
\begin{equation}\label{Eq_Holomorphic_decomposition}
f(u+Jv)=\alpha(u,v)+J\beta(u,v),\quad\Big(\text{resp.}\;f(u+Jv)=\alpha(u,v)+\beta(u,v)J\Big),
\end{equation}
for every $(u,v)\in\mathcal{U}$, $J\in\mathbb{S}$, and if the functions $\alpha,\beta$ satisfy the symmetry conditions
\begin{equation*}
\alpha(u,-v)=\alpha(u,v)\quad\text{and}\quad\beta(u,-v)=- \beta(u,v),
\end{equation*}
as well as the Cauchy-Riemann equations
\begin{equation}\label{Eq_Cauchy_Riemann_equations}
\frac{\partial}{\partial u}\alpha(u,v)=\frac{\partial}{\partial v}\beta(u,v),\quad\text{and}\quad\frac{\partial}{\partial v}\alpha(u,v)=-\frac{\partial}{\partial u}\beta(u,v).
\end{equation}
The class of left (resp. right) slice hyperholomorphic functions on $U$ is denoted by $\mathcal{SH}_L(U)$ (resp. $\mathcal{SH}_R(U)$). In the special case that the functions $\alpha$ and $\beta$ are real valued, we call the function $f$ \textit{intrinsic} and denote the space of intrinsic functions by $\mathcal{N}(U)$.
\end{defi}

Next we introduce path integrals of slice hyperholomorphic functions. Since it is sufficient to consider paths embedded in only one complex plane $\mathbb{C}_J$, the idea is to reduce it to a classical complex path integral.

\begin{defi}\label{defi_Path_integral}
Let $U\subseteq\mathbb{H}$ be axially symmetric, $J\in\mathbb{S}$. Then for any $f\in\mathcal{SH}_R(U')$, $g\in\mathcal{SH}_L(U')$ for some open set $U'\supseteq\overline{U}$, we define the integral
\begin{equation*}
\int_{\partial(U\cap\mathbb{C}_J)}f(s)ds_Jg(s):=\int_a^bf(\gamma(t))\frac{\gamma'(t)}{J}g(\gamma(t))dt,
\end{equation*}
where $\gamma:[a,b]\rightarrow\mathbb{C}_J$ is a parametrization of the boundary $\partial(U\cap\mathbb{C}_J)$.
\end{defi}

From now on $V$ always denotes a two-sided Banach space over the quaternions $\mathbb{H}$. In the following we will specify the class of operators with commuting components, which will be of interest in this paper.

\begin{defi}[Operators with commuting components]
A right-linear closed operator $T:V\rightarrow V$ with a two-sided linear domain $\dom(T)$ is said to be an \textit{operator with commuting components}, if there exist two-sided linear operators $T_i:V\rightarrow V$, $i\in\{0,\dots,3\}$, with $\dom(T_i)=\dom(T)$, such that
\begin{equation*}
T=T_0+e_1T_1+e_2T_2+e_3T_3,
\end{equation*}
and the property that for all $i,j\in\{0,\dots,3\}$
\begin{equation*}
\dom(T^2)\subseteq\dom(T_iT_j)\quad\text{and}\quad T_iT_j=T_jT_i\text{ on }\dom(T^2).
\end{equation*}
We will denote the class of closed operators with commuting components by $\mathcal{KC}(V)$.
\end{defi}

For any operator $T\in\mathcal{KC}(V)$ we additionally define the \textit{conjugate operator}
\begin{equation*}
\overline{T}:=T_0-e_1T_1-e_2T_2-e_3T_3,
\end{equation*}
with $\dom(\overline{T}):=\dom(T)$. Then for every $T\in\mathcal{KC}(V)$ there obviously holds the properties
\begin{equation*}
\dom(\overline{T}T)=\dom(T^2)\subseteq\dom(T\overline{T}),
\end{equation*}
as well as
\begin{equation*}
|T|^2:=\overline{T}T=T\overline{T}=T_0^2+T_1^2+T_2^2+T_3^2,\quad\text{on }\dom(T^2).
\end{equation*}
For any operator $T\in\mathcal{KC}(V)$ we now use the operator
\begin{equation}\label{Eq_Commutative_Q_resolvent}
Q_{c,s}(T):=s^2\mathcal{I}-2T_0s+|T|^2,\qquad\text{with }\dom Q_{c,s}(T)=\dom(T^2),
\end{equation}
to define the \textit{$S$-resolvent set}
\begin{equation}\label{Eq_Resolvent_set}
\rho_S(T):=\Set{s\in\mathbb{H} | Q_{c,s}(T)\text{ is bijective}}.
\end{equation}
The complement of this set will be called the \textit{$S$-spectrum}
\begin{equation*}
\sigma_S(T):=\mathbb{H}\setminus\rho_S(T).
\end{equation*}
It is proven in \cite[Theorem 3.1.6]{FJBOOK} that $\rho_S(T)$ is an axially symmetric open subset of $\mathbb{H}$ and that the mapping $s\mapsto Q_{c,s}^{-1}(T)$ is intrinsic on $\rho_S(T)$. Moreover, motivated by the Cauchy kernels \eqref{Eq_Cauchy_kernel} we define for every $s\in\rho_S(T)$ the \textit{left} and the \textit{right $S$-resolvent}
\begin{equation}\label{Eq_S_resolvent}
S_L^{-1}(s,T):=(s\mathcal{I}-\overline{T})Q_{c,s}^{-1}(T)\quad\text{and}\quad S_R^{-1}(s,T):=sQ_{c,s}^{-1}(T)-\sum\limits_{i=0}^3T_iQ_{c,s}^{-1}(T)\overline{e_i}.
\end{equation}
Note, that on $\dom(T)$ we are allowed to interchange $T_iQ_{c,s}^{-1}(T)=Q_{c,s}^{-1}(T)T_i$, which gives the more elegant form of the right $S$-resolvent
\begin{equation}\label{Eq_S_resolvent_on_domT}
S_R^{-1}(s,T)=Q_{c,s}^{-1}(T)(s\mathcal{I}-\overline{T}),\quad\text{on }\dom(T).
\end{equation}

\begin{rem}
Note, that the original definition of the $S$-resolvent set is
\begin{equation}\label{Eq_Resolvent_set_noncommuting}
\rho_S(T):=\Set{s\in\mathbb{H} | Q_s(T)\text{ is bijective}},
\end{equation}
using the operator
\begin{equation*}
Q_s(T):=T^2-2s_0T+|s|^2.
\end{equation*}
This definition has the advantage that it is well defined for any closed operator $T:V\rightarrow V$, not only for those with commuting components. However, it is proven in \cite[Theorem 3.3.4]{FJBOOK} that for operators $T\in\mathcal{KC}(V)$ the two definitions \eqref{Eq_Resolvent_set} and \eqref{Eq_Resolvent_set_noncommuting} coincide. Moreover, also the left (resp. right) $S$-resolvent operators \eqref{Eq_S_resolvent} are originally defined as
\begin{equation}\label{Eq_S_resolvent_noncommuting}
S_L^{-1}(s,T):=Q_s^{-1}(T)\overline{s}-TQ_s^{-1}(T)\quad\text{and}\quad S_R^{-1}(s,T):=(\overline{s}\mathcal{I}-T)Q_s^{-1}(T),
\end{equation}
but turn out to be equivalent to \eqref{Eq_S_resolvent} in the case of operators $T\in\mathcal{KC}(V)$.
\end{rem}

In order to finally define the class of operators, for which we will introduce the harmonic $H^\infty$-functional calculus in this paper, we denote for every angle $\omega\in(0,\pi)$ the open sector
\begin{equation*}
S_\omega:=\Set{s\in\mathbb{H}\setminus\{0\} | \vert\Arg(s)\vert<\omega},
\end{equation*}
where $\Arg(s)=0$ if $\Im(s)=0$, and in the case $\Im(s)\neq 0$, understood as the argument of a complex number in the complex plane $s\in\mathbb{C}_J$, with $J=\frac{\Im(s)}{|\Im(s)|}$.

\begin{defi}[Operators of type $\omega$]
Any $T\in\mathcal{KC}(V)$ is called \textit{operator of type} $\omega\in(0,\pi)$, if $\sigma_S(T)\subseteq\overline{S_\omega}$ and if for every $\theta\in(\omega,\pi)$ there exists some $C_\theta\geq 0$, such that
\begin{equation}\label{Eq_S_resolvent_estimates}
\Vert S_L^{-1}(s,T)\Vert\leq\frac{C_\theta}{|s|}\quad\text{and}\quad\Vert S_R^{-1}(s,T)\Vert\leq\frac{C_\theta}{|s|},\quad\text{for every }s\in S_\theta^c\setminus\{0\}.
\end{equation}
\end{defi}

A similar estimate as in \cite[Lemma 3.7]{CG18_2} now shows, that the estimates \eqref{Eq_S_resolvent_estimates} on the $S$-resolvent operators imply a similar estimate on the commutative $Q$-resolvent $Q_{c,s}^{-1}(T)$, which will be crucial for the existence of the integral \eqref{Eq_Q_functional_calculus_decaying} of the harmonic functional calculus.

\begin{lem}
Let $T\in\mathcal{KC}(V)$ such that $T,\overline{T}$ are of type $\omega$. Then for every $\theta\in(\omega,\pi)$ there exists some $C_\theta\geq 0$, such that
\begin{equation}\label{Eq_Q_resolvent_estimate}
\Vert Q_{c,s}^{-1}(T)\Vert\leq\frac{C_\theta}{|s|^2},\quad\text{for every }s\in S_\theta^c\setminus\{0\}.
\end{equation}
\end{lem}

\begin{proof}
Using the representations \eqref{Eq_S_resolvent} for $S_L^{-1}(s,T)$ and \eqref{Eq_S_resolvent_on_domT} for $S_R^{-1}(s,T)$, gives
\begin{align*}
\big(S_R^{-1}(s,T)+S_R^{-1}(s,\overline{T})\big)&\big(S_L^{-1}(s,T)+S_L^{-1}(s,\overline{T})\big) \\
&=Q_{c,s}^{-1}(T)\big((s\mathcal{I}-\overline{T})+(s\mathcal{I}-T)\big)\big((s\mathcal{I}-\overline{T})+(s\mathcal{I}-T)\big)Q_{c,s}^{-1}(T) \\
&=4Q_{c,s}^{-1}(T)(s\mathcal{I}-T_0)^2Q_{c,s}^{-1}(T),
\end{align*}
as well as
\begin{align*}
\big(S_R^{-1}(s,T)-S_R^{-1}(s,\overline{T})\big)&\big(S_L^{-1}(s,T)-S_L^{-1}(s,\overline{T})\big) \\
&=Q_{c,s}^{-1}(T)\big((s\mathcal{I}-\overline{T})-(s\mathcal{I}-T)\big)\big((s\mathcal{I}-\overline{T})-(s\mathcal{I}-T)\big)Q_{c,s}^{-1}(T) \\
&=4Q_{c,s}^{-1}(T)(T_0^2-|T|^2)Q_{c,s}^{-1}(T).
\end{align*}
Note, that it is justified by $\ran(S_L^{-1}(s,T))\subseteq\dom(T)$ to use the representation \eqref{Eq_S_resolvent_on_domT} of $S_R^{-1}(T)$. By subtracting these two equations we then conclude the identity
\begin{align*}
\big(S_R^{-1}(s,T)+&S_R^{-1}(s,\overline{T})\big)\big(S_L^{-1}(s,T)+S_L^{-1}(s,\overline{T})\big) \\
&-\big(S_R^{-1}(s,T)-S_R^{-1}(s,\overline{T})\big)\big(S_L^{-1}(s,T)-S_L^{-1}(s,\overline{T})\big) \\
&\hspace{2cm}=4Q_{c,s}^{-1}(T)(s^2-2sT_0+|T|^2)Q_{c,s}^{-1}(T)=4Q_{c,s}^{-1}(T).
\end{align*}
Since $T$ as well as $\overline{T}$ are assumed to be of type $\omega$, the estimates \eqref{Eq_S_resolvent_estimates} of the $S$-resolvents hold for $T$ as well as for $\overline{T}$ (let's say for the same constant $C_\theta$) and we get the final estimate
\begin{equation*}
\Vert Q_{c,s}^{-1}(T)\Vert\leq\frac{1}{4}\bigg(\bigg(\frac{C_\theta}{|s|}+\frac{C_\theta}{|s|}\bigg)\bigg(\frac{C_\theta}{|s|}+\frac{C_\theta}{|s|}\bigg)+\bigg(\frac{C_\theta}{|s|}+\frac{C_\theta}{|s|}\bigg)\bigg(\frac{C_\theta}{|s|}+\frac{C_\theta}{|s|}\bigg)\bigg)=\frac{2C_\theta^2}{|s|^2}. \qedhere
\end{equation*}
\end{proof}

\section{Harmonic functional calculus for decaying functions}\label{sec_Harmonic_functional_calculus_for_decaying_functions}

In this section we take the first step in establishing the harmonic functional calculus
for unbounded operators of type $\omega$, by giving a direct meaning to the integral \eqref{Eq_Q_functional_calculus_formal} for the following classes of slice hyperholomorphic functions

\begin{enumerate}
\item[i)] $\Psi_L^Q(S_\theta):=\Set{f\in\mathcal{SH}_L(S_\theta) | \exists\alpha>0,\,C_\alpha\geq 0: |f(s)|\leq\frac{C_\alpha|s|^{1+\alpha}}{1+|s|^{1+2\alpha}}\text{ for every }s\in S_\theta}$,
\item[ii)] $\Psi_R^Q(S_\theta):=\Set{f\in\mathcal{SH}_R(S_\theta) | \exists\alpha>0,\,C_\alpha\geq 0: |f(s)|\leq\frac{C_\alpha|s|^{1+\alpha}}{1+|s|^{1+2\alpha}}\text{ for every }s\in S_\theta}$,
\item[iii)] $\Psi^Q(S_\theta):=\Set{f\in\mathcal{N}(S_\theta) | \exists\alpha>0,\,C_\alpha\geq 0: |f(s)|\leq\frac{C_\alpha|s|^{1+\alpha}}{1+|s|^{1+2\alpha}}\text{ for every }s\in S_\theta}$.
\end{enumerate}

The decay at $0$ and at $\infty$ is necessary in order to make the integrals \eqref{Eq_S_functional_calculus_decaying} and \eqref{Eq_Q_functional_calculus_decaying} converge.

\begin{rem}\label{rem_Psi_spaces}
Since the $S$-resolvents $S_L^{-1}(s,T)$ and $S_R^{-1}(s,T)$ admit at $s=0$ only a $\frac{1}{|s|}$-singularity due to \eqref{Eq_S_resolvent_estimates}, instead of the $\frac{1}{|s|^2}$-singularity of $Q_{c,s}^{-1}(s,T)$ in \eqref{Eq_Q_resolvent_estimate}, the $\mathcal{O}(|s|^{1+\alpha})$ decay in the space $\Psi^Q(S_\theta)$ can be reduced to an $\mathcal{O}(|s|^\alpha)$ decay for the $S$-functional calculus \eqref{Eq_S_functional_calculus_decaying}. This means classically, the $S$-functional calculus is defined for the larger class of functions

\begin{enumerate}
\item[i)] $\Psi_L(S_\theta):=\Set{f\in\mathcal{SH}_L(S_\theta) | \exists\alpha>0,\,C_\alpha\geq 0: |f(s)|\leq\frac{C_\alpha|s|^\alpha}{1+|s|^{1+2\alpha}}\text{ for every }s\in S_\theta}$,
\item[ii)] $\Psi_R(S_\theta):=\Set{f\in\mathcal{SH}_R(S_\theta) | \exists\alpha>0,\,C_\alpha\geq 0: |f(s)|\leq\frac{C_\alpha|s|^\alpha}{1+|s|^{1+2\alpha}}\text{ for every }s\in S_\theta}$,
\item[iii)] $\Psi(S_\theta):=\Set{f\in\mathcal{N}(S_\theta) | \exists\alpha>0,\,C_\alpha\geq 0: |f(s)|\leq\frac{C_\alpha|s|^\alpha}{1+|s|^{1+2\alpha}}\text{ for every }s\in S_\theta}$.
\end{enumerate}

However, since we will use the $S$-functional typically in combination with the harmonic functional calculus \eqref{Eq_Q_functional_calculus_decaying}, we will consider functions in the smaller spaces $\Psi_L^Q(S_\theta)$, $\Psi_R^Q(S_\theta)$, $\Psi^Q(S_\theta)$ most of the time.
\end{rem}

The following well known $S$-functional calculus is motivated by the Cauchy formula \eqref{Eq_Cauchy_formula} and for example worked out in \cite{ACQS2016}.

\begin{defi}[$S$-functional calculus for decaying functions]\label{defi_S_functional_calculus_decaying}
Let $T\in\mathcal{KC}(V)$ be an operator of type $\omega$. Then for any $f\in\Psi_L(S_\theta)$ (resp. $f\in\Psi_R(S_\theta)$), $\theta\in(\omega,\pi)$, the $S$\textit{-functional calculus} is defined as the bounded, everywhere defined operator
\begin{subequations}
\begin{align}
&f(T):=\frac{1}{2\pi}\int_{\partial(S_\varphi\cap\mathbb{C}_J)}S_L^{-1}(s,T)ds_Jf(s), \label{Eq_S_functional_calculus_decaying} \\
&\bigg(\text{resp.}\;f(T):=\frac{1}{2\pi}\int_{\partial(S_\varphi\cap\mathbb{C}_J)}f(s)ds_JS_R^{-1}(s,T)\bigg), \label{Eq_S_functional_calculus_decaying_right}
\end{align}
\end{subequations}
where $\varphi\in(\omega,\theta)$ and $J\in\mathbb{S}$ are arbitrary and the integral is independent of those parameters.
\end{defi}

The novelty of this paper is now the following harmonic functional calculus, which replaces the left and right $S$-resolvent operators in \eqref{Eq_S_functional_calculus_decaying} by the commutative pseudo resolvent $Q_{c,s}^{-1}(T)$. This leads to a functional calculus for the axially harmonic functions $\mathcal{AH}(U)$, i.e. for $Df(T)$ for any slice hyperholomorphic function $f\in\mathcal{SH}(U)$. Note, that simply plugging $Df$ into the $S$-functional calculus \eqref{Eq_S_functional_calculus_decaying} is not allowed since in general $Df$ is no longer a slice hyperholomorphic function.

\begin{defi}[Harmonic functional calculus for decaying functions]\label{defi_Q_functional_calculus_decaying}
Let $T\in\mathcal{KC}(V)$ with $T,\overline{T}$ being operators of type $\omega$. Then for every $f\in\Psi_L^Q(S_\theta)$ (resp. $f\in\Psi_R^Q(S_\theta)$), $\theta\in(\omega,\pi)$, the \textit{harmonic functional calculus} is defined as

\medskip

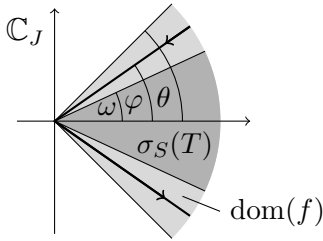
\begin{minipage}{0.3\textwidth}
\begin{center}
\begin{tikzpicture}
\fill[black!15] (1.56,1.56)--(0,0)--(1.56,-1.56) arc (-45:45:2.2);
\fill[black!30] (2,0.93)--(0,0)--(2,-0.93) arc (-25:25:2.2);
\draw (2,0.93)--(0,0)--(2,-0.93);
\draw (1.56,1.56)--(0,0)--(1.56,-1.56);
\draw (1.6,0) node[anchor=north] {$\sigma_S(T)$};
\draw (0.9,0) arc (0:25:0.9) (0.7,-0.05) node[anchor=south] {$\omega$};
\draw (1.3,0) arc (0:35:1.3) (1.05,-0.03) node[anchor=south] {$\varphi$};
\draw (1.7,0) arc (0:45:1.7) (1.45,0.05) node[anchor=south] {$\theta$};
\draw[thick] (1.8,-1.26)--(0,0)--(1.8,1.26);
\draw[thick,->] (1.8,1.26)--(1.47,1.03);
\draw[thick,->] (0,0)--(1.47,-1.03);
\draw[->] (-0.5,0)--(2.6,0);
\draw[->] (0,-1.5)--(0,1.5) node[anchor=north east] {\large{$\mathbb{C}_J$}};
\draw (1.7,-1)--(2.2,-1.2) node[anchor=west] {$\dom(f)$};
\end{tikzpicture}
\end{center}
\end{minipage}
\begin{minipage}{0.69\textwidth}
\begin{subequations}
\begin{align}
&Df(T):=\frac{-1}{\pi}\int_{\partial(S_\varphi\cap\mathbb{C}_J)}Q_{c,s}^{-1}(T)ds_Jf(s), \label{Eq_Q_functional_calculus_decaying} \\
&\bigg(\text{resp.}\;(fD)(T):=\frac{-1}{\pi}\int_{\partial(S_\varphi\cap\mathbb{C}_J)}f(s)ds_JQ_{c,s}^{-1}(T)\bigg), \label{Eq_Q_functional_calculus_decaying_right}
\end{align}
\end{subequations}
\hfill where $\varphi\in(\omega,\theta)$ and $J\in\mathbb{S}$ are arbitrary.
\end{minipage}
\end{defi}

The following theorem shows that the Definition~\ref{defi_Q_functional_calculus_decaying} is well posed.

\begin{thm}\label{thm_Q_functional_calculus_decaying}
Let $T\in\mathcal{KC}(V)$ such that $T,\overline{T}$ are of type $\omega$. Then for every $f\in\Psi_L^Q(S_\theta)$ (resp. $f\in\Psi_R^Q(S_\theta)$), $\theta\in(\omega,\pi)$, the integrals \eqref{Eq_Q_functional_calculus_decaying} are absolute convergent and do neither depend on the angle $\varphi\in(\omega,\theta)$ nor on the choice of the imaginary unit $J\in\mathbb{S}$. Moreover, if two functions $f,g\in\Psi_L^Q(S_\theta)$ (resp. $f,g\in\Psi_R^Q(S_\theta)$) satisfy $Df=Dg$ (resp. $fD=gD$), then also the functional calculi $Df(T)=Dg(T)$ (resp. $(fD)(T)=(gD)(T)$) coincide.
\end{thm}

\begin{proof}
We will only consider $f\in\Psi_L^Q(S_\theta)$, i.e. the first integral in \eqref{Eq_Q_functional_calculus_decaying}, the calculations for the second integral are the same.

\medskip

For the \textit{absolute convergence of  integral} \eqref{Eq_Q_functional_calculus_decaying}, we use the path $\gamma(t):=\begin{cases} -te^{J\varphi}, & t<0, \\ te^{-J\varphi}, & t>0, \end{cases}$ along the boundary of $S_\varphi\cap\mathbb{C}_J$. Then the estimate \eqref{Eq_Q_resolvent_estimate} of the operator $Q_{c,s}^{-1}(T)$ and the fact that $f\in\Psi_L^Q(S_\theta)$, give the absolute convergence of integral
\begin{equation}\label{Eq_Q_functional_calculus_decaying_15}
\int_{\mathbb{R}\setminus\{0\}}\big\Vert Q_{c,\gamma(t)}^{-1}(T)\big\Vert\,|\gamma'(t)|\,|f(\gamma(t))|dt\leq C_\varphi C_\alpha\int_{\mathbb{R}\setminus\{0\}}\frac{|t|^{\alpha-1}}{1+|t|^{1+2\alpha}}dt<\infty.
\end{equation}
For the \textit{independence of the angle} $\varphi$, we consider two angles $\varphi_1<\varphi_2\in(\omega,\theta)$ and for every $0<\varepsilon<R$ we consider the curves

\medskip

\begin{minipage}{0.29\textwidth}
\begin{center}
\begin{tikzpicture}
\fill[black!15] (1.25,2.17)--(0,0)--(1.25,-2.17) arc (-60:60:2.5);
\draw (1.25,2.17)--(0,0)--(1.25,-2.17);
\fill[black!30] (2.27,1.06)--(0,0)--(2.27,-1.06) arc (-25:25:2.5);
\draw (2.27,1.06)--(0,0)--(2.27,-1.06);
\draw[->] (-0.3,0)--(2.9,0);
\draw[->] (0,-2)--(0,2.3) node[anchor=north east] {\large{$\mathbb{C}_J$}};
\draw (1,0) arc (0:35:1) (0.75,-0.09) node[anchor=south] {\small{$\varphi_1$}};
\draw (1.4,0) arc (0:50:1.4) (1.15,0.06) node[anchor=south] {\small{$\varphi_2$}};
\draw[thick] (0.45,0.54)--(1.45,1.73) arc (50:35:2.26)--(0.58,0.41);
\draw[thick] (0.45,-0.54)--(1.45,-1.73) arc (-50:-35:2.26)--(0.58,-0.41);
\draw[thick,->] (0.58,0.41)--(0.45,0.54);
\draw (0.7,0.55) node[anchor=north east] {\small{$\sigma_\varepsilon$}};
\draw[thick,->] (0.45,-0.54)--(0.58,-0.41);
\draw (0.65,-0.55) node[anchor=south east] {\small{$\sigma_\varepsilon$}};
\draw[thick,->] (1.85,1.3)--(1.28,0.89);
\draw (1.28,0.78) node[anchor=west] {\small{$\gamma_{1,\varepsilon,R}$}};
\draw[thick,->] (0.58,-0.41)--(1.28,-0.89);
\draw (1.2,-0.7) node[anchor=west] {\small{$\gamma_{1,\varepsilon,R}$}};
\draw[thick,->] (1.45,1.73)--(1,1.2);
\draw (1.1,1.3) node[anchor=east] {\small{$\gamma_{2,\varepsilon,R}$}};
\draw[thick,->] (0.45,-0.54)--(1,-1.2);
\draw (1,-1.3) node[anchor=east] {\small{$\gamma_{2,\varepsilon,R}$}};
\draw[thick,->] (1.63,1.57)--(1.6,1.6);
\draw (1.6,1.6) node[anchor=west] {\small{$\sigma_R$}};
\draw[thick,->] (1.71,-1.48)--(1.73,-1.45);
\draw (1.65,-1.6) node[anchor=west] {\small{$\sigma_R$}};
\end{tikzpicture}
\end{center}
\end{minipage}
\begin{minipage}{0.7\textwidth}
\begin{align*}
\sigma_\varepsilon(\varphi)&:=\varepsilon e^{J\varphi},\hspace{1.03cm}\varphi\in(-\varphi_2,-\varphi_1)\cup(\varphi_1,\varphi_2), \\
\sigma_R(\varphi)&:=Re^{J\varphi},\hspace{0.9cm}\varphi\in(-\varphi_2,-\varphi_1)\cup(\varphi_1,\varphi_2), \\
\gamma_{1,\varepsilon,R}(t)&:=\begin{cases} -te^{J\varphi_1}, & t\in(-R,-\varepsilon), \\ te^{-J\varphi_1}, & t\in(\varepsilon,R), \end{cases} \\
\gamma_{2,\varepsilon,R}(t)&:=\begin{cases} -te^{J\varphi_2}, & t\in(-R,-\varepsilon), \\ te^{-J\varphi_2}, & t\in(\varepsilon,R). \end{cases}
\end{align*}
\end{minipage}

\medskip

Then the Cauchy integral theorem gives
\begin{equation}\label{Eq_Q_functional_calculus_decaying_1}
\int_{\gamma_{1,\varepsilon,R}}Q_{c,s}^{-1}(T)ds_Jf(s)=\int_{\sigma_R\oplus\gamma_{2,\varepsilon,R}\ominus\sigma_\varepsilon}Q_{c,s}^{-1}(T)ds_Jf(s).
\end{equation}
In the limit $\varepsilon\rightarrow 0^+$, the integral along $\sigma_\varepsilon$ vanishes, because of
\begin{align}
\bigg|\int_{\sigma_\varepsilon}Q_{c,s}^{-1}(T)ds_Jf(s)\bigg|&=\bigg|\int_{\varphi_1<|\varphi|<\varphi_2}Q_{c,\varepsilon e^{J\varphi}}^{-1}(T)\varepsilon e^{J\varphi}f(\varepsilon e^{J\varphi})d\varphi\bigg| \notag \\
&\leq C_{\varphi_1}C_\alpha\int_{\varphi_1<|\varphi|<\varphi_2}\frac{1}{\varepsilon^2}\,\varepsilon\,\frac{\varepsilon^{1+\alpha}}{1+\varepsilon^{1+2\alpha}}d\varphi \notag \\
&=2C_{\varphi_1}C_\alpha(\varphi_2-\varphi_1)\frac{\varepsilon^\alpha}{1+\varepsilon^{1+2\alpha}}\overset{\varepsilon\rightarrow 0^+}{\longrightarrow}0. \label{Eq_Q_functional_calculus_decaying_2}
\end{align}
In the same way also the integral along $\sigma_{R}$ vanishes in the limit $R\rightarrow\infty$,
\begin{equation}\label{Eq_Q_functional_calculus_decaying_3}
\bigg|\int_{\sigma_R}Q_{c,s}^{-1}(T)ds_Jf(s)\bigg|\leq 2C_{\varphi_1}C_\alpha(\varphi_2-\varphi_1)\frac{R^\alpha}{1+R^{1+2\alpha}}\overset{R\rightarrow\infty}{\longrightarrow}0.
\end{equation}
Performing now the limits $\varepsilon\rightarrow 0^+$ and $R\rightarrow\infty$ in \eqref{Eq_Q_functional_calculus_decaying_1} and using that the integrals \eqref{Eq_Q_functional_calculus_decaying_2} and \eqref{Eq_Q_functional_calculus_decaying_3} vanish, we obtain the independency of the angle
\begin{equation*}
\int_{\partial(S_{\varphi_1}\cap\mathbb{C}_J)}Q_{c,s}^{-1}(T)ds_Jf(s)=\int_{\partial(S_{\varphi_2}\cap\mathbb{C}_J)}Q_{s,c}^{-1}(T)ds_Jf(s).
\end{equation*}
For the \textit{independence on the imaginary unit} $J\in\mathbb{S}$, we consider two imaginary units $J,I\in\mathbb{S}$. For any angles $\varphi_1<\varphi_2<\varphi_3\in(\omega,\theta)$ we fix now $\varepsilon>0$ and define the paths

\medskip

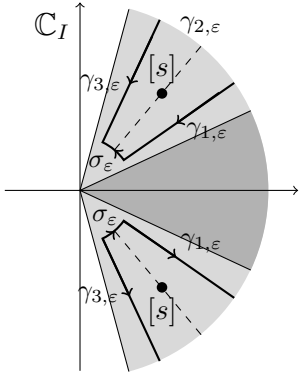
\begin{minipage}{0.29\textwidth}
\begin{center}
\begin{tikzpicture}
\fill[black!15] (0.65,2.41)--(0,0)--(0.65,-2.41) arc (-75:75:2.5);
\draw (0.65,2.41)--(0,0)--(0.65,-2.41);
\fill[black!30] (2.27,1.06)--(0,0)--(2.27,-1.06) arc (-25:25:2.5);
\draw (2.27,1.06)--(0,0)--(2.27,-1.06);
\draw[->] (-1,0)--(2.9,0);
\draw[->] (0,-2.4)--(0,2.5) node[anchor=north east] {\large{$\mathbb{C}_I$}};
\draw[thick] (1.06,2.27)--(0.3,0.64) arc (65:35:0.71)--(2.05,1.43);
\draw[thick] (1.06,-2.27)--(0.3,-0.64) arc (-65:-35:0.71)--(2.05,-1.43);
\draw[thick,->] (0.47,0.53)--(0.46,0.54);
\draw (0.6,0.6) node[anchor=north east] {\small{$\sigma_\varepsilon$}};
\draw[thick,->] (0.45,-0.55)--(0.46,-0.54);
\draw (0.65,-0.6) node[anchor=south east] {\small{$\sigma_\varepsilon$}};
\draw[thick,->] (1.85,1.3)--(1.28,0.89);
\draw (1.28,0.78) node[anchor=west] {\small{$\gamma_{1,\varepsilon}$}};
\draw[thick,->] (0.58,-0.41)--(1.28,-0.89);
\draw (1.2,-0.7) node[anchor=west] {\small{$\gamma_{1,\varepsilon}$}};
\draw[thick,->] (0.96,2.05)--(0.66,1.41);
\draw (0.69,1.41) node[anchor=east] {\small{$\gamma_{3,\varepsilon}$}};
\draw[thick,->] (0.3,-0.64)--(0.66,-1.41);
\draw (0.66,-1.41) node[anchor=east] {\small{$\gamma_{3,\varepsilon}$}};
\draw[dashed] (1.61,-1.92)--(0.46,-0.54);
\draw[dashed] (0.46,0.54)--(1.61,1.92) node[anchor=south] {\small{$\gamma_{2,\varepsilon}$}};
\fill[black] (1.09,1.29) circle (0.07cm) node[anchor=south] {$[s]$};
\fill[black] (1.09,-1.29) circle (0.07cm) node[anchor=north] {$[s]$};
\end{tikzpicture}
\end{center}
\end{minipage}
\begin{minipage}{0.7\textwidth}
\begin{align}
\gamma_{1,\varepsilon}(t)&:=\begin{cases} -te^{I\varphi_1}, & t<-\varepsilon, \\ te^{-I\varphi_1}, & t>\varepsilon, \end{cases} \notag \\
\gamma_{2,\varepsilon}(t)&:=\begin{cases} -te^{J\varphi_2}, & t<-\varepsilon, \\ te^{-J\varphi_2}, & t>\varepsilon, \end{cases} \label{Eq_Q_functional_calculus_decaying_9} \\
\gamma_{3,\varepsilon}(t)&:=\begin{cases} -te^{I\varphi_3}, & t<-\varepsilon, \\ te^{-I\varphi_3}, & t>\varepsilon, \end{cases} \notag \\
\sigma_\varepsilon(\varphi)&:=\varepsilon e^{I\varphi},\quad \varphi\in(-\varphi_3,-\varphi_1)\cup(\varphi_1,\varphi_3). \notag
\end{align}
Note, that $\gamma_{1,\varepsilon}$, $\gamma_{3,\varepsilon}$, $\sigma_\varepsilon$ are curves in $\mathbb{C}_I$, while $\gamma_{2,\varepsilon}$ is in $\mathbb{C}_J$.
\end{minipage}

\medskip

The Cauchy formula \eqref{Eq_Cauchy_formula} then gives the representation
\begin{equation}\label{Eq_Q_functional_calculus_decaying_4}
f(s)=\frac{1}{2\pi}\int_{\gamma_{3,\varepsilon}\ominus\sigma_\varepsilon\ominus\gamma_{1,\varepsilon}}S_L^{-1}(p,s)dp_If(p),\qquad s\in\ran(\gamma_{2,\varepsilon}),
\end{equation}
where the integral along $\infty$, which closes the path on the right, vanishes due to the asymptotics
\begin{equation}\label{Eq_Q_functional_calculus_decaying_8}
|S_L^{-1}(p,s)f(p)|=\frac{|p-\overline{s}|}{|p^2-2s_0p+|s|^2|}\frac{C_\alpha|p|^{1+\alpha}}{1+|p|^{1+2\alpha}}=\mathcal{O}(|p|^{-1-\alpha}),\quad\text{as }p\rightarrow\infty.
\end{equation}
Next we consider the curves

\medskip

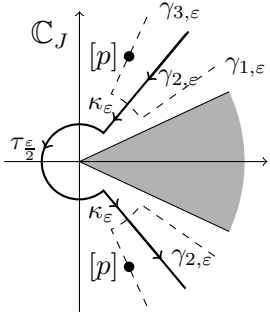
\begin{minipage}{0.29\textwidth}
\begin{center}
\begin{tikzpicture}
\fill[black!30] (1.99,0.93)--(0,0)--(1.99,-0.93) arc (-25:25:2.2);
\draw (1.99,0.93)--(0,0)--(1.99,-0.93);
\draw[->] (-1,0)--(2.5,0);
\draw[->] (0,-2)--(0,2) node[anchor=north east] {\large{$\mathbb{C}_J$}};
\fill[black] (0.66,1.4) circle (0.07cm) node[anchor=east] {$[p]$};
\fill[black] (0.66,-1.4) circle (0.07cm) node[anchor=east] {$[p]$};
\draw[dashed] (1.8,1.26)--(0.82,0.57) arc (35:65:1)--(0.93,1.99) node[anchor=west] {\small{$\gamma_{3,\varepsilon}$}};
\draw[dashed] (1.8,-1.26)--(0.82,-0.57) arc (-35:-65:1)--(0.93,-1.99);
\draw (1.8,1.26) node[anchor=west] {\small{$\gamma_{1,\varepsilon}$}};
\draw[thick] (1.41,1.69)--(0.32,0.38) arc (50:310:0.5)--(1.41,-1.69);
\draw[thick,->] (0.64,-0.77)--(1.09,-1.3);
\draw (1.09,-1.3) node[anchor=west] {\small{$\gamma_{2,\varepsilon}$}};
\draw[thick,->] (1.45,1.73)--(0.91,1.08);
\draw (0.91,1.1) node[anchor=west] {\small{$\gamma_{2,\varepsilon}$}};
\draw[thick,->] (0.51,0.61)--(0.45,0.54);
\draw (0.6,0.75) node[anchor=east] {\small{$\kappa_\varepsilon$}};
\draw[thick,->] (0.45,-0.54)--(0.51,-0.61);
\draw (0.6,-0.7) node[anchor=east] {\small{$\kappa_\varepsilon$}};
\draw[thick,->] (-0.45,0.19)--(-0.46,0.17);
\draw (-0.41,0.17) node[anchor=east] {\small{$\tau_{\frac{\varepsilon}{2}}$}};
\end{tikzpicture}
\end{center}
\end{minipage}
\begin{minipage}{0.7\textwidth}
\begin{align}
\tau_{\frac{\varepsilon}{2}}(\varphi)&:=\frac{\varepsilon}{2}e^{J\varphi},\qquad \varphi\in(\varphi_2,2\pi-\varphi_2), \notag \\
\kappa_\varepsilon(t)&:=\begin{cases} -te^{J\varphi_2}, & t\in(-\varepsilon,-\frac{\varepsilon}{2}), \\ te^{-J\varphi_2}, & t\in(\frac{\varepsilon}{2},\varepsilon). \end{cases} \label{Eq_Q_functional_calculus_decaying_10}
\end{align}
In this setting the Cauchy formula \eqref{Eq_Cauchy_formula} gives
\begin{align}
Q_{c,p}^{-1}(T)&=\frac{-1}{2\pi}\int_{\gamma_{2,\varepsilon}\oplus\kappa_\varepsilon\oplus\tau_{\frac{\varepsilon}{2}}}Q_{c,s}^{-1}(T)ds_JS_R^{-1}(s,p) \notag \\
&\hspace{-1cm}=\frac{1}{2\pi}\int_{\gamma_{2,\varepsilon}\oplus\kappa_\varepsilon\oplus\tau_{\frac{\varepsilon}{2}}}Q_{c,s}^{-1}(T)ds_JS_L^{-1}(p,s),\quad p\in\ran(\gamma_{3,\varepsilon}), \label{Eq_Q_functional_calculus_decaying_6}
\end{align}
\end{minipage}

\medskip

where the negative sign in the first line comes from the fact that the curve $\gamma_{2,\varepsilon}\oplus\kappa_\varepsilon\oplus\tau_{\frac{\varepsilon}{2}}$ (closed at $\infty$ on the left) surrounds the points $[p]\cap\mathbb{C}_J$ in the negative sense, and in the second line we used the connection $S_R^{-1}(s,p)=-S_L^{-1}(p,s)$ between the left and the right Cauchy kernel. The integral along $\infty$, which closes the path on the left, vanishes due to the asymptotics
\begin{equation*}
\Vert Q_{c,s}^{-1}(T)\Vert\,|S_R^{-1}(s,p)|\leq\frac{C_{\varphi_2}}{|s|^2}\frac{|s-\overline{p}|}{|s^2-2p_0s+|p|^2|}=\mathcal{O}(|s|^{-2}),\quad\text{as }s\rightarrow\infty.
\end{equation*}
Analogously we obtain
\begin{equation}\label{Eq_Q_functional_calculus_decaying_5}
0=\frac{1}{2\pi}\int_{\gamma_{2,\varepsilon}\oplus\kappa_\varepsilon\oplus\tau_{\frac{\varepsilon}{2}}}Q_{c,s}^{-1}(T)ds_JS_L^{-1}(p,s),\qquad p\in\ran(\gamma_{1,\varepsilon}),
\end{equation}
since in this case the points $[p]\cap\mathbb{C}_J$ lie outside the integration path and hence the Cauchy integral vanishes. The combination of \eqref{Eq_Q_functional_calculus_decaying_4}, \eqref{Eq_Q_functional_calculus_decaying_6} and \eqref{Eq_Q_functional_calculus_decaying_5}, leads to the formula
\begin{align}
\int_{\gamma_{2,\varepsilon}}Q_{c,s}^{-1}(T)ds_Jf(s)&=\frac{1}{2\pi}\int_{\gamma_{2,\varepsilon}}Q_{c,s}^{-1}(T)ds_J\bigg(\int_{\gamma_{3,\varepsilon}\ominus\sigma_\varepsilon\ominus\gamma_{1,\varepsilon}}S_L^{-1}(p,s)dp_If(p)\bigg) \notag \\
&=\int_{\gamma_{3,\varepsilon}}\bigg(Q_{c,p}^{-1}(T)-\frac{1}{2\pi}\int_{\kappa_\varepsilon\oplus\tau_{\frac{\varepsilon}{2}}}Q_{c,s}^{-1}(T)ds_JS_L^{-1}(p,s)\bigg)dp_If(p) \notag \\
&\quad-\frac{1}{2\pi}\int_{\sigma_\varepsilon}\bigg(\int_{\gamma_{2,\varepsilon}}Q_{c,s}^{-1}(T)ds_JS_L^{-1}(p,s)\bigg)dp_If(p) \notag \\
&\quad+\frac{1}{2\pi}\int_{\gamma_{1,\varepsilon}}\bigg(\int_{\kappa_\varepsilon\oplus\tau_{\frac{\varepsilon}{2}}}Q_{c,s}^{-1}(T)ds_JS_L^{-1}(p,s)\bigg)dp_If(p) \notag \\
&\hspace{-2cm}=\int_{\gamma_{3,\varepsilon}}Q_{c,p}^{-1}(T)dp_If(p)-\frac{1}{2\pi}\int_{\sigma_\varepsilon}\int_{\gamma_{2,\varepsilon}\oplus\kappa_\varepsilon\oplus\tau_{\frac{\varepsilon}{2}}}Q_{c,s}^{-1}(T)ds_JS_L^{-1}(p,s)dp_If(p), \label{Eq_Q_functional_calculus_decaying_7}
\end{align}
where in the last equation we combined the integration path $\gamma_{1,\varepsilon}\ominus\gamma_{3,\varepsilon}$ and replaced it by $\sigma_\varepsilon$, see also the graphic in \eqref{Eq_Q_functional_calculus_decaying_9}. This is allowed by the Cauchy theorem since the integral along $\infty$, which closes the path on the right, vanishes due the asymptotics \eqref{Eq_Q_functional_calculus_decaying_8}, and since the singularities of $S_L^{-1}(p,s)$ lie on $\kappa_\varepsilon\oplus\tau_{\frac{\varepsilon}{2}}$, which is outside the integration path, see the graphic in \eqref{Eq_Q_functional_calculus_decaying_10}. Finally, we now perform the limit $\varepsilon\rightarrow 0^+$ to this equation and show that the double integral in \eqref{Eq_Q_functional_calculus_decaying_7} vanishes. Therefore, we estimate the integrand by
\begin{equation*}
\big|Q_{c,s}^{-1}(T)S_L^{-1}(p,s)f(p)\big|\leq\frac{C_{\varphi_1}}{|s|^2}\frac{|p|+|s|}{|p^2-2s_0p+|s|^2|}\frac{C_\alpha|p|^{1+\alpha}}{1+|p|^{1+2\alpha}}\leq\frac{C_{\varphi_1}C_\alpha(|p|+|s|)|p|^{1+\alpha}}{|s|^2|p-s_j||p-\overline{s_j}|},
\end{equation*}
where $\{s_I,\overline{s_I}\}=[s]\cap\mathbb{C}_I$ are the intersections of $[s]$ with the complex plane $\mathbb{C}_I$. The first part of the integral then vanishes because of
\begin{align*}
\bigg|\int_{\sigma_\varepsilon}\int_{\gamma_{2,\varepsilon}\oplus\kappa_\varepsilon}&Q_{c,s}^{-1}(T)ds_JS_L^{-1}(p,s)dp_If(p)\bigg| \\
&\leq 2C_{\varphi_1}C_\alpha\int_{\varphi_1<|\varphi|<\varphi_3}\int_{\frac{\varepsilon}{2}}^\infty\frac{(\varepsilon+r)\varepsilon^{1+\alpha}}{r^2|\varepsilon e^{I\varphi}-re^{I\varphi_2}||\varepsilon e^{I\varphi}-re^{-I\varphi_2}|}\varepsilon drd\varphi \\
&=2C_{\varphi_1}C_\alpha\varepsilon^\alpha\int_{\varphi_1<|\varphi|<\varphi_3}\int_{\frac{1}{2}}^\infty\frac{1+\rho}{\rho^2|e^{I\varphi}-\rho e^{I\varphi_2}||e^{I\varphi}-\rho e^{-I\varphi_2}|}d\rho d\varphi\overset{\varepsilon\rightarrow 0^+}{\longrightarrow}0,
\end{align*}
where the double integral in the last line exists since the singularities $(\rho,\varphi)=(1,\pm\varphi_2)$ of the denominator is of order $1$, which is integrable in the two-dimensional integral. Also the second part of the double integral in \eqref{Eq_Q_functional_calculus_decaying_7} vanishes, since
\begin{align*}
\bigg|\int_{\sigma_\varepsilon}\int_{\tau_{\frac{\varepsilon}{2}}}&Q_{c,s}^{-1}(T)ds_JS_L^{-1}(p,s)dp_If(p)\bigg| \\
&\leq C_{\varphi_1}C_\alpha\int_{\varphi_1<|\varphi|\varphi_3}\int_{\varphi_2}^{2\pi-\varphi_2}\frac{(\varepsilon+\frac{\varepsilon}{2})\varepsilon^{1+\alpha}}{\frac{\varepsilon^2}{4}|\varepsilon e^{I\varphi}-\frac{\varepsilon}{2}e^{I\phi}||\varepsilon e^{I\varphi}-\frac{\varepsilon}{2}e^{-I\phi}|}\varepsilon\frac{\varepsilon}{2}d\phi d\varphi \\
&=3C_{\varphi_1}C_\alpha\varepsilon^\alpha\int_{\varphi_1<|\varphi|<\varphi_3}\int_{\varphi_2}^{2\pi-\varphi_2}\frac{1}{|e^{I\varphi}-\frac{1}{2}e^{I\phi}||e^{I\varphi}-\frac{1}{2}e^{-I\phi}|}d\phi d\varphi\overset{\varepsilon\rightarrow 0^+}{\longrightarrow}0.
\end{align*}
Hence, in the limit $\varepsilon\rightarrow 0^+$ the equation \eqref{Eq_Q_functional_calculus_decaying_7} turns into the desired independence of the imaginary unit
\begin{equation*}
\int_{\partial(S_{\varphi_2}\cap\mathbb{C}_J)}Q_{c,s}^{-1}(T)ds_Jf(s)=\int_{\partial(S_{\varphi_3}\cap\mathbb{C}_I)}Q_{c,p}^{-1}(T)dp_If(p).
\end{equation*}
For the \textit{independence of the kernel of} $D$, as in \cite{CDP23}, we consider two functions $f,g\in\mathcal{SH}_L(S_\theta)$ with $Df(s)=Dg(s)$. Due to the decomposition \eqref{Eq_Holomorphic_decomposition} we can write
\begin{equation*}
f(u+Jv)-g(u+Jv)=\alpha(u,v)+J\beta(u,v).
\end{equation*}
Moreover, we can write the Cauchy-Fueter operator in spherical coordinates with respect to this decomposition $x=u+Jv$, with $u\in\mathbb{R}$ and $v>0$, as
\begin{equation*}
D=\frac{\partial}{\partial x_0}+e_1\frac{\partial}{\partial x_1}+e_2\frac{\partial}{\partial x_2}+e_3\frac{\partial}{\partial x_3}=\frac{\partial}{\partial u}+J\frac{\partial}{\partial v}+\frac{J}{v}\Gamma_J,
\end{equation*}
where $\Gamma_J$ is a symbol for the angular derivatives. Using now the identity $J\Gamma_JJ=\Gamma_J-2$, see \cite[Paragraph 1.12.1]{DSS92}, it follows from $D(f-g)(s)=0$ that the functions $\alpha$ and $\beta$ satisfy the Vekua-type differential equations
\begin{align*}
\frac{\partial}{\partial u}\alpha(u,v)-\frac{\partial}{\partial v}\beta(u,v)&=\frac{2}{v}\beta(u,v), \\
\frac{\partial}{\partial v}\alpha(u,v)+\frac{\partial}{\partial u}\beta(u,v)&=0.
\end{align*}
However, $\alpha$ and $\beta$ also satisfy the Cauchy-Riemann equations \eqref{Eq_Cauchy_Riemann_equations} since $f-g$ is slice hyperholomorphic, and hence we conclude $\beta(u,v)=0$ and consequently also $\frac{\partial}{\partial u}\alpha(u,v)=\frac{\partial}{\partial v}\alpha(u,v)=0$ have to vanish. Since the domain $S_\theta$ on which the function $f-g$ is holomorphic is connected, this implies, that the function
\begin{equation}\label{Eq_Q_functional_calculus_decaying_16}
f(s)-g(s)=c\quad\text{is constant for every }s\in S_\theta\text{ with }\Im(s)\neq 0.
\end{equation}
By continuity, the difference $f(s)-g(s)=c$ then has to be constant for every $s\in S_\theta$. However, this is only possible for $c=0$, since otherwise $f-g\notin\Psi_L^Q(S_\theta)$. This means that $f=g$ and then $Df(T)=Dg(T)$ follows from the integral \eqref{Eq_Q_functional_calculus_decaying}.
\end{proof}

In the next lemma we state same basic properties of the functional calculus \eqref{Eq_Q_functional_calculus_decaying}.

\begin{lem}\label{lem_Properties_Q_functional_calculus_decaying}
Let $T\in\mathcal{KC}(V)$ with $T,\overline{T}$ being operators of type $\omega$, $f,g\in\Psi^Q_L(S_\theta)$, for some $\theta\in(\omega,\pi)$. Then

\medskip

\begin{enumerate}
\item[i)] $Df(T)$ is a bounded operator with commuting components.
\item[ii)] $D(f+g)(T)=Df(T)+Dg(T)$,
\item[iii)] $Df(\overline{T})=Df(T)$,
\item[iv)] If $f$ is intrinsic, then $\overline{Df(T)}=Df(T)$.
\end{enumerate}

The same results hold true for $f,g\in\Psi_R^Q(S_\theta)$ and the corresponding functional calculus \eqref{Eq_Q_functional_calculus_decaying_right}.
\end{lem}

\begin{proof}
The boundedness of $Df(T)$ in the statement i) follows from the estimate \eqref{Eq_Q_functional_calculus_decaying_15}. The linearity in ii) follows from the linearity of the integral \eqref{Eq_Q_functional_calculus_decaying}. In order to show that $Df(T)$ has commuting components, we fix $J\in\mathbb{S}$ and $\varphi\in(\omega,\pi)$ and write the integral \eqref{Eq_Q_functional_calculus_decaying} as
\begin{align}
Df(T)&=\frac{-1}{\pi}\int_{\partial(S_\varphi\cap\mathbb{C}_J)}Q_{c,s}^{-1}(T)ds_Jf(s) \notag \\
&=\frac{-1}{\pi}\bigg(\int_{-\infty}^0Q_{c,-te^{J\varphi}}^{-1}(T)Je^{J\varphi}f(-te^{J\varphi})dt-\int_0^\infty Q_{c,te^{-J\varphi}}^{-1}(T)Je^{-J\varphi}f(te^{-J\varphi})dt\bigg) \notag \\
&=\frac{-1}{\pi}\int_0^\infty\Big(Q_{c,te^{J\varphi}}^{-1}(T)Je^{J\varphi}f(te^{J\varphi})-Q_{c,te^{-J\varphi}}^{-1}(T)Je^{-J\varphi}f(te^{-J\varphi})\Big)dt \notag \\
&=\frac{-1}{\pi}\int_0^\infty\Big(Q_{c,te^{-J\varphi}}(T)Je^{J\varphi}f(te^{J\varphi})-Q_{c,te^{J\varphi}}(T)Je^{-J\varphi}f(te^{-J\varphi})\Big)|Q_{c,te^{J\varphi}}(T)|^{-2}dt, \label{Eq_Properties_Q_functional_calculus_decaying_5}
\end{align}
where in the last line we rewrote the commutative pseudo resolvent as
\begin{equation*}
Q_{c,s}^{-1}(T)=Q_{c,\overline{s}}(T)Q_{c,\overline{s}}^{-1}(T)Q_{c,s}^{-1}(T)=Q_{c,\overline{s}}(T)|Q_{c,s}(T)|^{-2},
\end{equation*}
with
\begin{equation}\label{Eq_Properties_Q_functional_calculus_decaying_3}
|Q_{c,s}(T)|^2=|s|^4\mathcal{I}-4s_0|s|^2T_0+4s_0^2|T|^2-2|s|^2|T|^2-4s_0T_0|T|^2+4|s|^2T_0^2+|T|^4.
\end{equation}
Using also the explicit value
\begin{equation}\label{Eq_Properties_Q_functional_calculus_decaying_6}
Q_{c,te^{\pm J\varphi}}(T)=t^2e^{\pm 2J\varphi}-2te^{\pm J\varphi}T_0+|T|^2
\end{equation}
from \eqref{Eq_Commutative_Q_resolvent} as well as the decomposition
\begin{equation}\label{Eq_Properties_Q_functional_calculus_decaying_1}
f(te^{J\varphi})=\alpha(t\cos\varphi,t\sin\varphi)+J\beta(t\cos\varphi,t\sin\varphi),
\end{equation}
from \eqref{Eq_Holomorphic_decomposition}, makes this integral
\begin{align}
Df(T)&=\frac{-1}{\pi}\int_0^\infty\Big(\big(Q_{c,te^{-J\varphi}}(T)e^{J\varphi}-Q_{c,te^{J\varphi}}(T)e^{-J\varphi}\big)J\alpha(t\cos\varphi,t\sin\varphi) \notag \\
&\hspace{2.3cm}-\big(Q_{c,te^{-J\varphi}}(T)e^{J\varphi}+Q_{c,te^{J\varphi}}(T)e^{-J\varphi} \big)\beta(t\cos\varphi,t\sin\varphi)\Big)|Q_{c,te^{J\varphi}}(T)|^{-2}dt \notag \\
&=\frac{2}{\pi}\int_0^\infty\Big((|T|^2-t^2)\sin\varphi\alpha(t\cos\varphi,t\sin\varphi) \notag \\
&\hspace{2.5cm}+\big((|T|^2+t^2)\cos\varphi-2tT_0\big)\beta(t\cos\varphi,t\sin\varphi)\Big)|Q_{c,te^{J\varphi}}(T)|^{-2}dt. \label{Eq_Properties_Q_functional_calculus_decaying_4}
\end{align}
If we now decompose
\begin{equation*}
\alpha(u,v)=\alpha_0(u,v)+\sum\limits_{i=1}^3e_i\alpha_i(u,v)\qquad\text{and}\qquad\beta(u,v)=\beta_0(u,v)+\sum\limits_{i=1}^3e_i\beta_i(u,v),
\end{equation*}
into their components, the components $Df(T)_i$, $i\in\{0,\dots,3\}$ of the harmonic functional calculus are explicitly given by
\begin{align}
Df(T)_i&=\frac{2}{\pi}\int_0^\infty\Big((|T|^2-t^2)\sin\varphi\alpha_i(t\cos\varphi,t\sin\varphi) \notag \\
&\hspace{2.3cm}+\big((|T|^2+t^2)\cos\varphi-2tT_0\big)\beta_i(t\cos\varphi,t\sin\varphi)\Big)|Q_{c,te^{J\varphi}}(T)|^{-2}dt, \label{Eq_Properties_Q_functional_calculus_decaying_2}
\end{align}
which are obviously commuting, since the components of $T$ are commuting and there are no imaginary units involved.

\medskip

The statement ii) follows from the fact, that the resolvent $Q_{c,s}^{-1}(T)=Q_{c,s}^{-1}(\overline{T})$ is the same for $T$ and for $\overline{T}$ by definition \eqref{Eq_Commutative_Q_resolvent}. Hence also $Dg(T)=Dg(\overline{T})$ follows immediately.

\medskip

For the proof of iii) we note that by Definition~\ref{defi_Slice_hyperholomorphic_functions}, the functions $\alpha$ and $\beta$ in the decomposition \eqref{Eq_Properties_Q_functional_calculus_decaying_1} are real valued since $f$ is intrinsic. Then it follows from the explicit representation \eqref{Eq_Properties_Q_functional_calculus_decaying_2} that $Df(T)_1=Df(T)_2=Df(T)_3=0$. Hence only the term $Df(T)=Df(T)_0$ remains and the conjugate operator is $\overline{Df(T)}=Df(T)$. \qedhere
\end{proof}

To the best of our knowledge, a similar statement as Lemma \ref{lem_Properties_Q_functional_calculus_decaying} is not known for the $S$-functional calculus \eqref{Eq_S_functional_calculus_decaying}. However, we will need those results in Section \ref{sec_Harmonic_functional_calculus_for_growing_functions}, and hence prove the following lemma regarding properties of the commuting components of the $S$-functional calculus.

\begin{lem}\label{lem_Properties_S_functional_calculus_decaying}
Let $T\in\mathcal{KC}(V)$ with $T,\overline{T}$ being operators of type $\omega$, and $f,g\in\Psi_L(S_\theta)$ or $f,g\in\Psi_R(S_\theta)$, for some $\theta\in(\omega,\pi)$. Then

\medskip

\begin{enumerate}
\item[i)] $f(T)$ is a bounded operator with commuting components.
\item[ii)] $(f+g)(T)=f(T)+g(T)$
\item[iii)] If $f$ is intrinsic, then $\overline{f(T)}=f(\overline{T})$.
\end{enumerate}
\end{lem}

\begin{proof}
The boundedness i) of the $S$-functional calculus \eqref{Eq_S_functional_calculus_decaying} follows from the estimate
\begin{equation*}
\int_{\mathbb{R}\setminus\{0\}}\big\Vert S_L^{-1}(\gamma(t),T)\big\Vert\,|\gamma'(t)|\,|f(\gamma(t))|dt\leq C_\varphi C_\alpha\int_{\mathbb{R}\setminus\{0\}}\frac{|t|^{\alpha-1}}{1+|t|^{1+2\alpha}}dt<\infty,
\end{equation*}
which is similar to \eqref{Eq_Q_functional_calculus_decaying_15}, using the bound \eqref{Eq_S_resolvent_estimates} of the $S$-resolvent and the definition of the space $\Psi_L(S_\theta)$ in Remark \ref{rem_Psi_spaces}. The linearity ii) is also clear by the linearity of the integral \eqref{Eq_S_functional_calculus_decaying}. For the proof of iii) we fix $J\in\mathbb{S}$, $\varphi\in(\omega,\pi)$ and derive in a similar way as in \eqref{Eq_Properties_Q_functional_calculus_decaying_5} with $Q_{c,s}^{-1}(T)$ replaced by $S_L^{-1}(s,T)=(s \mathcal{I}-\overline{T})Q_{c,s}^{-1}(T)$ the formula
\begin{align}
f(T)&=\frac{1}{2\pi}\int_0^\infty\Big((te^{J\varphi}-\overline{T})Q_{c,te^{-J\varphi}}(T)Je^{J\varphi}f(te^{J\varphi}) \notag \\
&\hspace{2.3cm}-(te^{-J\varphi}-\overline{T})Q_{c,te^{J\varphi}}(T)Je^{-J\varphi}f(te^{-J\varphi})\Big)|Q_{c,te^{J\varphi}}(T)|^{-2}dt. \label{Eq_Properties_S_functional_calculus_decaying_1}
\end{align}
Using the explicit value \eqref{Eq_Properties_Q_functional_calculus_decaying_6} of $Q_{c,te^{\pm J\varphi}}(T)$ and the decomposition \eqref{Eq_Properties_Q_functional_calculus_decaying_1} of $f$, we further rewrite the integral \eqref{Eq_Properties_S_functional_calculus_decaying_1} in the same as we derived \eqref{Eq_Properties_Q_functional_calculus_decaying_3}, such that we get
\begin{align}
f(T)&=\frac{1}{\pi}\int_0^\infty\Big(\big(2tT_0\sin\varphi-t|T|^2\sin(2\varphi)\big)\alpha(t\cos\varphi,t\sin\varphi) \notag \\
&\hspace{2.3cm}-\big(t^2\sin\varphi-|T|^2\sin\varphi\big)\overline{T}\,\alpha(t\cos\varphi,t\sin\varphi) \notag \\
&\hspace{2.3cm}-\big(t^3-2t^2T_0\cos\varphi+t|T|^2\cos(2\varphi)\big)\beta(t\cos\varphi,t\sin\varphi) \notag \\
&\hspace{2.3cm}+\big(t^2\cos\varphi-2tT_0+|T|^2\cos\varphi\big)\overline{T}\,\beta(t\cos\varphi,t\sin\varphi)\Big)|Q_{c,te^{J\varphi}}(T)|^{-2}dt. \label{Eq_Properties_S_functional_calculus_decaying_2}
\end{align}
If we finally write
\begin{equation*}
\alpha=\alpha_0+\sum\limits_{i=1}^3e_i\alpha_i,\quad\beta=\beta_0+\sum\limits_{i=1}^3e_i\beta_i\quad\text{and}\quad\overline{T}=T_0-\sum\limits_{i=1}^3e_iT_i,
\end{equation*}
with real valued functions $\alpha_i,\beta_i:\mathcal{U}\rightarrow\mathbb{R}$, $i\in\{0,1,2,3\}$, and sort \eqref{Eq_Properties_S_functional_calculus_decaying_2} with respect to the imaginary units, we end up with components of $f(T)$ which are real linear combinations of powers of $T_i$, $i\in\{0,1,2,3\}$. Due to the commutativity of the components $T_i$, also the components of $f(T)$ are then commuting. Hence i) is proven. For the proof of iii), we note that $\alpha$ and $\beta$ are real valued, since $f$ is assumed to be intrinsic. Hence we immediately obtain the identity $\overline{f(T)}=f(\overline{T})$ from \eqref{Eq_Properties_S_functional_calculus_decaying_2}.
\end{proof}

\begin{prop}\label{prop_Commutation_with_Df}
Let $T\in\mathcal{KC}(V)$ with $T,\overline{T}$ being operators of type $\omega$, $B:V\rightarrow V$ an everywhere defined bounded operator and $f\in\Psi^Q(S_\theta)$, for some $\theta\in(\omega,\pi)$. Then there holds
\begin{equation*}
\begin{array}{c} \forall j\in\{0,\dots,3\}: BT_j=T_jB\text{ on }\dom(T) \\ \Downarrow \\ BDf(T)=Df(T)B. \end{array}
\end{equation*}
\end{prop}

\begin{proof}
From the assumption it follows that $B$ in particular commutes with the term $|Q_{c,s}(T)|^2$ from \eqref{Eq_Properties_Q_functional_calculus_decaying_3} on $\dom(T^2)$, and hence also with its inverse
\begin{equation*}
|Q_{c,s}(T)|^{-2}B=B|Q_{c,s}(T)|^{-2},\quad\text{on }V.
\end{equation*}
Using this fact in the representation \eqref{Eq_Properties_Q_functional_calculus_decaying_4} of $Df(T)$, immediately gives the commutation
\begin{equation*}
Df(T)B=BDf(T)
\end{equation*}
since there are no imaginary units in the integrand of \eqref{Eq_Properties_Q_functional_calculus_decaying_4}, not even in the functions $\alpha$ and $\beta$, which are real valued since $f$ is assumed to be intrinsic.
\end{proof}

\begin{cor}\label{cor_Commutation_g_with_Df}
Let $T\in\mathcal{KC}(V)$ with $T,\overline{T}$ being operators of type $\omega$, $f\in\Psi^Q(S_\theta)$ and $g\in\Psi_L^Q(S_\theta)$ or $g\in\Psi_R^Q(S_\theta)$, for some $\theta\in(\omega,\pi)$. Then there holds
\begin{equation*}
g(T)Df(T)=Df(T)g(T).
\end{equation*}
\end{cor}

\begin{proof}
We only consider $g\in\Psi_L^Q(S_\theta)$, the case $g\in\Psi_R^Q(S_\theta)$ follows analogously. It is clear by definition \eqref{Eq_S_resolvent}, that $T_jS_L^{-1}(s,T)=S_L^{-1}(s,T)T_j$ commutes on $\dom(T)$ for every $j\in\{0,\dots,3\}$ and hence also
\begin{equation*}
g(T)T_j=\frac{1}{2\pi}\int_{\partial(S_\varphi\cap\mathbb{C}_J)}S_L^{-1}(s,T)ds_Jg(s)T_j=\frac{1}{2\pi}\int_{\partial(S_\varphi\cap\mathbb{C}_J)}T_jS_L^{-1}(s,T)ds_Jg(s)=T_jg(T),
\end{equation*}
on $\dom(T)$. From Proposition~\ref{prop_Commutation_with_Df} we then conclude $g(T)Df(T)=Df(T)g(T)$.
\end{proof}

Next we want to derive the very important product rule of the harmonic functional calculus. The basic ingredient will be the following resolvent equation.

\begin{lem}
Let $T\in\mathcal{KC}(V)$, $s,p\in\rho_S(T)$ with $s\notin[p]$. Then
\begin{subequations}\label{Eq_Q_resolvent_equation}
\begin{align}
Q_{c,s}^{-1}(T)S_L^{-1}(p,s)+S_R^{-1}(s,p)Q_{c,p}^{-1}(T)&=Q_{c,s}^{-1}(T)S_L^{-1}(p,T)+S_R^{-1}(s,\overline{T})Q_{c,p}^{-1}(T) \label{Eq_Q_resolvent_equation1} \\
&=Q_{c,s}^{-1}(T)S_L^{-1}(p,\overline{T})+S_R^{-1}(s,T)Q_{c,p}^{-1}(T). \label{Eq_Q_resolvent_equation2}
\end{align}
\end{subequations}
\end{lem}

\begin{proof}
Using the left $S$-resolvent in \eqref{Eq_S_resolvent} and the representation \eqref{Eq_S_resolvent_on_domT} of the right resolvent, we obtain
\begin{align}
Q_{c,s}^{-1}(T)S_L^{-1}&(p,T)+S_R^{-1}(s,\overline{T})Q_{c,p}^{-1}(T) \notag \\
&=Q_{c,s}^{-1}(T)\big(S_L^{-1}(p,T)Q_{c,p}(T)+Q_{c,s}(T)S_R^{-1}(s,\overline{T})\big)Q_{c,p}^{-1}(T) \notag \\
&=Q_{c,s}^{-1}(T)\big((p-\overline{T})Q_{c,p}^{-1}(T)Q_{c,p}(T)+Q_{c,s}(T)Q_{c,s}^{-1}(T)(s-T)\big)Q_{c,p}^{-1}(T) \notag \\
&=Q_{c,s}^{-1}(T)(p-2T_0+s)Q_{c,p}^{-1}(T).\label{Eq_Q_resolvent_equation_1}
\end{align}
Since the right hand side of this equation does not change when we replace $T\rightarrow\overline{T}$, the equation \eqref{Eq_Q_resolvent_equation2} is proven. Moreover, using \eqref{Eq_S_resolvent} for $S_L^{-1}(p,s)$ and \eqref{Eq_S_resolvent_noncommuting} for $S_R^{-1}(s,p)$, analogously gives
\begin{align}
Q_{c,s}^{-1}(T)S_L^{-1}&(p,s)+S_R^{-1}(s,p)Q_{c,p}^{-1}(T) \notag \\
&=Q_{c,s}^{-1}(T)\big(S_L^{-1}(p,s)Q_{c,p}(T)+Q_{c,s}(T)S_R^{-1}(s,p)\big)Q_{c,p}^{-1}(T) \notag \\
&=Q_{c,s}^{-1}(T)\big((p-\overline{s})Q_{c,p}(T)+Q_{c,s}(T)(\overline{s}-p)\big)(p^2-2s_0p+|s|^2)^{-1}Q_{c,p}^{-1}(T) \notag \\
&=Q_{c,s}^{-1}(T)(p-2T_0+s)Q_{c,p}^{-1}(T). \label{Eq_Q_resolvent_equation_2}
\end{align}
Comparison of \eqref{Eq_Q_resolvent_equation_1} and \eqref{Eq_Q_resolvent_equation_2} gives the resolvent identity \eqref{Eq_Q_resolvent_equation1}.
\end{proof}

With the resolvent equation \eqref{Eq_Q_resolvent_equation}, we are now ready to prove the product rule of the harmonic functional calculus.

\begin{thm}\label{thm_Product_rule_decaying}
Let $T\in\mathcal{KC}(V)$ with $T,\overline{T}$ being operators of type $\omega$ and $\theta\in(\omega,\pi)$. Then for every $f\in\Psi^Q(S_\theta)$, $g\in\Psi_L^Q(S_\theta)$ with $\theta\in(\omega,\pi)$ we obtain the product rule
\begin{subequations}\label{Eq_Product_rule_left_decaying}
\begin{align}
D(fg)(T)&=Df(T)g(T)+f(\overline{T})Dg(T) \label{Eq_Product_rule_left_decaying1} \\
&=Df(T)g(\overline{T})+f(T)Dg(T). \label{Eq_Product_rule_left_decaying2}
\end{align}
\end{subequations}
The same product rule also holds for $f\in\Psi^Q_R(S_\theta)$ and $g\in\Psi^Q(S_\theta)$ using the respective functional calculi \eqref{Eq_S_functional_calculus_decaying_right} and \eqref{Eq_Q_functional_calculus_decaying_right}.
\end{thm}

\begin{proof}
We only prove the product rule \eqref{Eq_Product_rule_left_decaying1}, since \eqref{Eq_Product_rule_left_decaying2} follows from \eqref{Eq_Product_rule_left_decaying1} when we replace $T$ by $\overline{T}$ and using Lemma~\ref{lem_Properties_Q_functional_calculus_decaying}~iii).

\medskip

Since $f\in\Psi^Q(S_\theta)$ is intrinsic and $g\in\Psi_L^Q(S_\theta)$ is left slice hyperholomorphic, the product $fg$ is also left slice hyperholomorphic and satisfies the needed estimates in order to be in the space $fg\in\Psi_L(S_\theta)$. Consider now two angles $\varphi_2<\varphi_1\in(\omega,\theta)$ and imaginary units $J,I\in\mathbb{S}$. Then, using the functional calculi \eqref{Eq_S_functional_calculus_decaying} and \eqref{Eq_Q_functional_calculus_decaying}, we have
\begin{align}
Df&(T)g(T)+f(\overline{T})Dg(T) \notag \\
&=\frac{-1}{2\pi^2}\bigg(\int_{\partial(S_{\varphi_1}\cap\mathbb{C}_J)}f(s)ds_JQ_{c,s}^{-1}(T)\int_{\partial(S_{\varphi_2}\cap\mathbb{C}_I)}S_L^{-1}(p,T)dp_Ig(p) \notag \\
&\hspace{2cm}+\int_{\partial(S_{\varphi_1}\cap\mathbb{C}_J)}f(s)ds_JS_R^{-1}(s,\overline{T})\int_{\partial(S_{\varphi_2}\cap\mathbb{C}_I)}Q_{c,p}^{-1}(T)dp_Ig(p)\bigg) \notag \\
&=\frac{-1}{2\pi^2}\int_{\partial(S_{\varphi_1}\cap\mathbb{C}_J)}f(s)ds_J\int_{\partial(S_{\varphi_2}\cap\mathbb{C}_I)}\Big(Q_{c,s}^{-1}(T)S_L^{-1}(p,T)+S_R^{-1}(s,\overline{T})Q_{c,p}^{-1}(T)\Big)dp_Ig(p) \notag \\
&=\frac{-1}{2\pi^2}\int_{\partial(S_{\varphi_1}\cap\mathbb{C}_J)}f(s)ds_J\int_{\partial(S_{\varphi_2}\cap\mathbb{C}_I)}\Big(Q_{c,s}^{-1}(T)S_L^{-1}(p,s)+S_R^{-1}(s,p)Q_{c,p}^{-1}(T)\Big)dp_Ig(p), \label{Eq_Product_rule_decaying_1}
\end{align}
where in the last line we used the resolvent equation \eqref{Eq_Q_resolvent_equation1}. Since we chose $\varphi_2<\varphi_1$, every point $s\in\partial(S_{\varphi_1}\cap\mathbb{C}_J)$ lies outside $\overline{S_{\varphi_2}}$ and hence we get
\begin{equation}\label{Eq_Product_rule_decaying_2}
\int_{\partial(S_{\varphi_2}\cap\mathbb{C}_I)}S_L^{-1}(p,s)dp_Ig(p)=0,\qquad s\in\partial(S_{\varphi_1}\cap\mathbb{C}_J),
\end{equation}
where the integral along $\infty$, which closes the set on the right, vanishes due to the asymptotics
\begin{equation*}
\big|S_L^{-1}(p,s)g(p)\big|\leq\frac{C_\beta|p-\overline{s}||p|^{1+\beta}}{|p^2-2s_0p+|s|^2|(1+|p|^{1+2\beta})}=\mathcal{O}\Big(\frac{1}{|p|^{1+\beta}}\Big),\qquad\text{as }p\rightarrow\infty.
\end{equation*}
Analogously, every point $p\in\partial(S_{\varphi_2}\cap\mathbb{C}_I)$ is contained in $S_{\varphi_1}$, which gives
\begin{equation}\label{Eq_Product_rule_decaying_3}
\frac{1}{2\pi}\int_{\partial(S_{\varphi_1}\cap\mathbb{C}_J)}f(s)ds_JS_R^{-1}(s,p)=f(p),\qquad p\in\partial(S_{\varphi_2}\cap\mathbb{C}_I).
\end{equation}
Plugging now \eqref{Eq_Product_rule_decaying_2} and \eqref{Eq_Product_rule_decaying_3} into \eqref{Eq_Product_rule_decaying_1} reduces the double integral to
\begin{align*}
Df(T)g(T)+f(\overline{T})Dg(T)&=\frac{-1}{\pi}\int_{\partial(S_{\varphi_2}\cap\mathbb{C}_I)}f(p)Q_{c,p}^{-1}(T)dp_Ig(p), \\
&=\frac{-1}{\pi}\int_{\partial(S_{\varphi_2}\cap\mathbb{C}_I)}Q_{c,p}^{-1}(T)dp_If(p)g(p)=D(fg)(T),
\end{align*}
where in the second line we used that $f$ is intrinsic and interchanges with $Q_{c,p}^{-1}(T)dp_J$.
\end{proof}

In the last part of this section we investigate how the harmonic functional calculus acts on intrinsic rational functions, i.e. on functions of the form $f(s)=\frac{p(s)}{q(s)}$, consisting of polynomials $p,q$ with real valued coefficients. To do so, we consider for every polynomial $q(s)=\sum_{j=0}^mb_js^j$ and every right-linear closed operator $T$ the \textit{polynomial functional calculus}
\begin{equation}\label{Eq_Polynomial_functional_calculus}
q[T]:=\sum\limits_{j=0}^mq_jT^j,
\end{equation}
with $\dom(q[T]):=\dom(T^m)$ and start with a lemma, ensuring its invertibility.

\begin{lem}\label{lem_Polynomial_bijective}
Let $T:V\rightarrow V$ be a right-linear closed operator with a two-sided linear domain and $q\not\equiv 0$ an intrinsic polynomial. If $q$ does not admit any zeros in $\sigma_S (T)$, then $q[T]$ is bijective.
\end{lem}

\begin{proof}
For the constant polynomial $q(s)=q_0$, we have $q_0\neq 0$ by the assumption $q\not\equiv 0$. Hence the operator $q[T]=q_0$ is bijective. Let us now consider $\deg(q)=m\geq 1$ and choose some arbitrary bounded slice Cauchy domain $U\subseteq\rho_S(T)$, containing all the zeros of $q$. Then define for some fixed $J\in\mathbb{S}$ the operator
\begin{equation}\label{Eq_Polynomial_bijective_1}
B:=\frac{-1}{2\pi}\int_{\partial(U\cap\mathbb{C}_J)}S_L^{-1}(s,T)ds_J\frac{1}{q(s)}.
\end{equation}
Then we  will prove by induction that $\ran(B)\subseteq\dom(T^j)$ and
\begin{equation}\label{Eq_Polynomial_bijective_2}
T^jB=\frac{-1}{2\pi}\int_{\partial(U\cap\mathbb{C}_J)}S_L^{-1}(s,T)ds_J\frac{s^j}{q(s)},\qquad j\in\{0,\dots,m-1\}.
\end{equation}
The induction start $j=0$ is clear by definition \eqref{Eq_Polynomial_bijective_1}. For the induction step $j-1\rightarrow j$ we use the identity
\begin{equation}\label{Eq_Polynomial_bijective_5}
S_L^{-1}(s,T)s=TS_L^{-1}(s,T)+\mathcal{I},
\end{equation}
see \cite[Theorem 2.33]{G17}, to get
\begin{equation}\label{Eq_Polynomial_bijective_3}
\int_{\partial(U\cap\mathbb{C}_J)}S_L^{-1}(s,T)ds_J\frac{s^j}{q(s)}=\int_{\partial(U\cap\mathbb{C}_J)}TS_L^{-1}(s,T)ds_J\frac{s^{j-1}}{q(s)}+\int_{\partial(U\cap\mathbb{C}_J)}ds_J\frac{s^{j-1}}{q(s)}.
\end{equation}
Since the last integral does no longer contain the $S$-resolvent operator $S_L^{-1}(s,T)$, it is holomorphic everywhere except the zeros of the denominator $q(s)$, which all lie inside $U$. Hence it is allowed by the Cauchy theorem to replace $U$ by a ball $B_{0,R}$ centered at the origin and with large enough radius $R$ such that it contains all zeros of $q$. This then leads to
\begin{equation}\label{Eq_Polynomial_bijective_6}
\int_{\partial(U\cap\mathbb{C}_J)}ds_J\frac{s^{j-1}}{q(s)}=\int_{\partial(B_{0,R}\cap\mathbb{C}_J)}ds_J\frac{s^{j-1}}{q(s)}\overset{R\rightarrow\infty}{\longrightarrow}0,
\end{equation}
where the right hand side vanishes in the limit $R\rightarrow\infty$ due to the asymptotics $\frac{s^{j-1}}{q(s)}=\mathcal{O}(|s|^{j-1-m})$ and $j\leq m-1$. Hence the left hand side, which does not depend on the radius $R$, has to vanish identically. This reduces equation \eqref{Eq_Polynomial_bijective_3} to
\begin{align*}
\int_{\partial(U\cap\mathbb{C}_J)}S_L^{-1}(s,T)ds_J\frac{s^j}{q(s)}&=\int_{\partial(U\cap\mathbb{C}_J)}TS_L^{-1}(s,T)ds_J\frac{s^{j-1}}{q(s)} \\
&=T\int_{\partial(U\cap\mathbb{C}_J)}S_L^{-1}(s,T)ds_J\frac{s^{j-1}}{q(s)}=-2\pi T^jB,
\end{align*}
where in the second equation we used Hille's theorem \cite[Theorem II.2.6]{DU77} and in the third equation the induction assumption \eqref{Eq_Polynomial_bijective_2} for $j-1$. Hence \eqref{Eq_Polynomial_bijective_2} is proven.

\medskip

Next, we once more use the resolvent identity \eqref{Eq_Polynomial_bijective_5} to write
\begin{align*}
S_L^{-1}(s,T)q(s)&=b_0S_L^{-1}(s,T)+\sum\limits_{j=1}^mb_j(TS_L^{-1}(s,T)+\mathcal{I})s^{j-1} \\
&=b_0S_L^{-1}(s,T)+T\sum\limits_{j=1}^mb_jS_L^{-1}(s,T)s^{j-1}+\frac{q(s)-b_0}{s}.
\end{align*}
By the Cauchy theorem and the already proven identity \eqref{Eq_Polynomial_bijective_2}, we then have
\begin{align}
0&=\frac{-1}{2\pi}\int_{\partial(U\cap\mathbb{C}_J)}S_L^{-1}(s,T)ds_J1=\frac{-1}{2\pi}\int_{\partial(U\cap\mathbb{C}_J)}S_L^{-1}(s,T)ds_J\frac{q(s)}{q(s)} \notag \\
&=b_0B+T\sum\limits_{j=1}^mb_jT^{j-1}B-\frac{1}{2\pi}\int_{\partial(U\cap\mathbb{C}_J)}ds_J\frac{q(s)-b_0}{sq(s)} \notag \\
&=q[T]B-\frac{1}{2\pi}\int_{\partial(U\cap\mathbb{C}_J)}ds_J\frac{q(s)-b_0}{sq(s)}. \label{Eq_Polynomial_bijective_7}
\end{align}
In the same way as in \eqref{Eq_Polynomial_bijective_6}, we are allowed to exchange $U$ by a ball $B_{0,R}$ centered at the origin and with a large enough radius to contain all zeros of $q$. Here it is important that $s=0$ is not an additional singularity of the integrand since it is also a zero of the nominator. This then leads to the integral
\begin{align*}
\frac{1}{2\pi}\int_{\partial(U\cap\mathbb{C}_J)}ds_J\frac{q(s)-b_0}{sq(s)}&=\frac{1}{2\pi}\int_{\partial(B_{0,R}\cap\mathbb{C}_J)}ds_J\frac{q(s)-b_0}{sq(s)} \\
&=1-\frac{b_0}{2\pi}\int_{\partial(B_{0,R}\cap\mathbb{C}_J)}ds_J\frac{1}{sq(s)}\overset{R\rightarrow\infty}{\longrightarrow}1,
\end{align*}
where we used that the last integral vanishes in the limit $R\rightarrow\infty$ since $\frac{1}{sq(s)}=\mathcal{O}(s^{-1-m})$ and $m\geq 1$. Plugging this into \eqref{Eq_Polynomial_bijective_7} gives $q[T]B=1$. Moreover, the operator $B$ commutes with $T$ on $\dom(T)$, because
\begin{align*}
TB&=\frac{-1}{2\pi}\int_{\partial(U\cap\mathbb{C}_J)}TS_L^{-1}(s,T)ds_J\frac{1}{q(s)}=\frac{-1}{2\pi}\int_{\partial(U\cap\mathbb{C}_J)}(S_L^{-1}(s,T)s-\mathcal{I})ds_J\frac{1}{q(s)} \\
&=\frac{-1}{2\pi}\int_{\partial(U\cap\mathbb{C}_J)}\frac{1}{q(s)}ds_J(sS_R^{-1}(s,T)-\mathcal{I})=\frac{-1}{2\pi}\int_{\partial(U\cap\mathbb{C}_J)}\frac{1}{q(s)}ds_JS_R^{-1}(s,T)T \\
&=\frac{-1}{2\pi}\int_{\partial(U\cap\mathbb{C}_J)}S_L^{-1}(s,T)ds_J\frac{1}{q(s)}T=BT,
\end{align*}
where we used the resolvent equation $sS_R^{-1}(s,T)-\mathcal{I}=S_R^{-1}(s,T)T$ from \cite[Theorem 2.33]{G17}, as well as the fact that for any intrinsic function $g$ the two integrals
\begin{equation*}
\int_{\partial(U\cap\mathbb{C}_J)}S_L^{-1}(s,T)ds_Jg(s)=\int_{\partial(U\cap\mathbb{C}_J)}g(s)ds_JS_R^{-1}(s,T)
\end{equation*}
coincide. With this commutativity we can rewrite the already proven $q[T]B=1$ also as $Bq[T]=1$ on $\dom(T^m)$, which proves that $q[T]$ is bijective with inverse
\begin{equation}\label{Eq_Polynomial_bijective_4}
q[T]^{-1}=B=\frac{-1}{2\pi}\int_{\partial(U\cap\mathbb{C}_J)}S_L^{-1}(s,T)ds_J\frac{1}{q(s)}. \qedhere
\end{equation}
\end{proof}

\begin{prop}\label{prop_Rational_equivalence_decaying}
Let $T\in\mathcal{KC}(V)$ with $T,\overline{T}$ being operators of type $\omega$ and consider two intrinsic polynomials $p,q$ with the properties

\begin{enumerate}
\item[i)] $\deg(q)\geq\deg(p)+1$;
\item[ii)] $p$ admits a zero of at least order $2$ at the origin;
\item[iii)] $q$ does not admit any zeros in $\overline{S_\omega}$.
\end{enumerate}

Then $q[T]$, $q[\overline{T}]$ are bijective, $\frac{p}{q}\in\Psi^Q(S_\theta)$ for some $\theta\in(\omega,\pi)$ and
\begin{subequations}\label{Eq_Rational_equivalence_decaying}
\begin{align}
D\Big(\frac{p}{q}\Big)(T)&=\big(Dp[T]q[T]-p[T]Dq[T]\big)q[T]^{-1}q[\overline{T}]^{-1} \label{Eq_Rational_equivalence_decaying1} \\
&=\big(Dp[T]q[\overline{T}]-p[\overline{T}]Dq[T]\big)q[T]^{-1}q[\overline{T}]^{-1}. \label{Eq_Rational_equivalence_decaying2}
\end{align}
\end{subequations}
\end{prop}

\begin{rem}
In \eqref{Eq_Rational_equivalence_decaying} the left hand side is understood as the harmonic functional calculus \eqref{Eq_Q_functional_calculus_decaying} (indicated by the round brackets), while the terms on the two right hand sides are understood as the polynomial functional calculus \eqref{Eq_Polynomial_functional_calculus} and as the \textit{harmonic polynomial functional calculus}
\begin{equation}\label{Eq_Derivative_polynomial_operator}
Dp[T]:=-2\sum\limits_{i=0}^np_i\sum\limits_{k=0}^{i-1}T^k\overline{T}^{i-1-k},
\end{equation}
with $\dom(Dp[T]):=\dom(T^{n-1})$, (indicated by square brackets). Note, that the operator \eqref{Eq_Derivative_polynomial_operator} is motivated by the action $Ds^i=-2\sum_{k=0}^{i-1}s^k\overline{s}^{i-1-k}$ of the Cauchy Fueter operator \eqref{Eq_Cauchy_Fueter_operator} on monomials, see \cite[Lemma 1]{B}.
\end{rem}

\begin{proof}[Proof of Proposition \ref{prop_Rational_equivalence_decaying}]
Note, that the second equation \eqref{Eq_Rational_equivalence_decaying2} follows from the first one \eqref{Eq_Rational_equivalence_decaying1} when we replace $T$ by $\overline{T}$ and use that all terms involving $D$ stay the same due to Lemma~\ref{lem_Properties_Q_functional_calculus_decaying}~iii) and the definitions \eqref{Eq_Polynomial_functional_calculus} and \eqref{Eq_Derivative_polynomial_operator} of the polynomial functional calculi. Hence it is left to prove the equality \eqref{Eq_Rational_equivalence_decaying1}.

\medskip

In the \textit{first step} we will change the integration path $\partial(S_\varphi\cap\mathbb{C}_J)$ in the integrals \eqref{Eq_S_functional_calculus_decaying} and \eqref{Eq_Q_functional_calculus_decaying} of $(\frac{p}{q})(T)$ and $D(\frac{p}{q})(T)$, to some finite path in $\overline{S_\omega}^c$ surrounding all the zeros of $q$. Since the closed set $\overline{S_\omega}$ does not contain any zeros of the polynomial $q$ by assumption iii), we can choose $\theta\in(\omega,\pi)$ small enough such that $\overline{S_\theta}$ does not contain any zeros of $q$ either. This means that $\frac{p}{q}$ is intrinsic on $S_\theta$. Moreover, $p$ admits a zeros of at least order $2$ at the origin by assumption ii), $q$ does not admit a zero at the origin by assumption iii) and $\deg(q)\geq\deg(p)+1$ by assumption i). These three properties ensure the existence of a constant $C_1\geq 0$ such that
\begin{equation}\label{Eq_Rational_estimate}
\Big|\frac{p(s)}{q(s)}\Big|\leq C_1\frac{|s|^2}{1+|s|^3},\quad s\in S_\theta,
\end{equation}
and hence $\frac{p}{q}\in\Psi^Q(S_\theta)$. By choosing $\varepsilon$ small and $R$ large enough, one can ensure that all the zeros of $q$ lie inside the set

\medskip

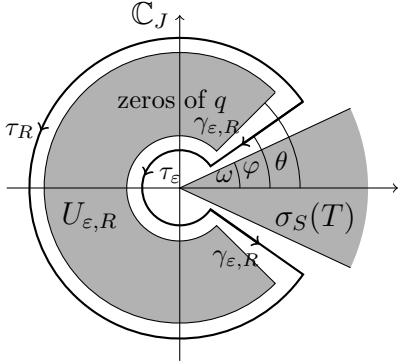
\begin{minipage}{0.39\textwidth}
\begin{center}
\begin{tikzpicture}
\fill[black!30] (2.27,1.06)--(0,0)--(2.27,-1.06) arc (-25:25:2.5);
\draw (2.27,1.06)--(0,0)--(2.27,-1.06);
\fill[black!30] (1.27,1.27)--(0.49,0.49) arc (45:315:0.7)--(1.27,-1.27) arc (315:45:1.8);
\draw (1.27,1.27)--(0.49,0.49) arc (45:315:0.7)--(1.27,-1.27) arc (315:45:1.8);
\draw (-0.1,1.15) node[anchor=center] {\small{zeros of $q$}};
\draw (0.8,0) arc (0:25:0.8) (0.6,-0.05) node[anchor=south] {\small{$\omega$}};
\draw (1.2,0) arc (0:35:1.2) (0.95,0) node[anchor=south] {\small{$\varphi$}};
\draw (1.6,0) arc (0:45:1.6) (1.35,0.1) node[anchor=south] {\small{$\theta$}};
\draw (1.8,-0.05) node[anchor=north] {\large{$\sigma_S(T)$}};
\draw (-1.2,-0.05) node[anchor=north] {\large{$U_{\varepsilon,R}$}};
\draw[->] (-2.3,0)--(2.9,0);
\draw[->] (0,-2.3)--(0,2.3) node[anchor=east] {\large{$\mathbb{C}_J$}};
\draw[thick] (1.64,1.15)--(0.41,0.29) arc (35:325:0.5)--(1.64,-1.15) arc (325:35:2);
\draw[thick,->] (1.64,1.15)--(0.82,0.57);
\draw(0.95,0.55) node[anchor=south east] {\small{$\gamma_{\varepsilon,R}$}};
\draw[thick,->] (0.41,-0.29)--(1.06,-0.75);
\draw(1.2,-0.7) node[anchor=north east] {\small{$\gamma_{\varepsilon,R}$}};
\draw[thick,->] (-0.45,0.19)--(-0.46,0.17);
\draw (-0.41,0.17) node[anchor=west] {\small{$\tau_\varepsilon$}};
\draw[thick,->] (-1.84,0.78)--(-1.85,0.76);
\draw (-1.8,0.76) node[anchor=east] {\small{$\tau_R$}};
\end{tikzpicture}
\end{center}
\end{minipage}
\begin{minipage}{0.6\textwidth}
\begin{equation*}
U_{\varepsilon,R}:=\Set{s\in\overline{S_\varphi}^c | \varepsilon<|s|<R}.
\end{equation*}
If we decompose $\partial(U_{\varepsilon,R}\cap\mathbb{C}_J)=\tau_R\ominus\gamma_{\varepsilon,R}\ominus\tau_\varepsilon$, using the curves
\begin{align*}
\gamma_{\varepsilon,R}(t)&:=\begin{cases} -te^{J\varphi}, & t\in(-R,-\varepsilon), \\ te^{-J\varphi}, & t\in(\varepsilon,R), \end{cases} \\
\tau_\varepsilon(\phi)&:=\varepsilon e^{J\phi},\hspace{0.8cm}\phi\in(\varphi,2\pi-\varphi), \\
\tau_R(\phi)&:=Re^{J\phi},\hspace{0.7cm}\phi\in(\varphi,2\pi-\varphi),
\end{align*}
\end{minipage}

\medskip

we note that the two integrals
\begin{equation*}
\lim\limits_{\varepsilon\rightarrow 0^+}\int_{\tau_\varepsilon}Q_{c,s}^{-1}(T)ds_J\frac{p(s)}{q(s)}=\lim\limits_{R\rightarrow\infty}\int_{\tau_R}Q_{c,s}^{-1}(T)ds_J\frac{p(s)}{q(s)}=0,
\end{equation*}
vanish due to the estimates \eqref{Eq_Q_resolvent_estimate} and \eqref{Eq_Rational_estimate}. Hence we obtain
\begin{align}
D\Big(\frac{p}{q}\Big)(T)&=\frac{-1}{\pi}\int_{\partial(S_\varphi\cap\mathbb{C}_J)}Q_{c,s}^{-1}(T)ds_J\frac{p(s)}{q(s)} \notag \\
&=\frac{-1}{\pi}\lim\limits_{\varepsilon\rightarrow 0}\lim\limits_{R\rightarrow\infty}\int_{\gamma_{\varepsilon,R}}Q_{c,s}^{-1}(T)ds_J\frac{p(s)}{q(s)} \notag \\
&=\frac{1}{\pi}\lim\limits_{\varepsilon\rightarrow 0}\lim\limits_{R\rightarrow\infty}\int_{\partial(U_{\varepsilon,R}\cap\mathbb{C}_J)}Q_{c,s}^{-1}(T)ds_J\frac{p(s)}{q(s)} \notag \\
&=\frac{1}{\pi}\int_{\partial(U_{\varepsilon,R}\cap\mathbb{C}_J)}Q_{c,s}^{-1}(T)ds_J\frac{p(s)}{q(s)}, \label{Eq_Rational_equivalence_1}
\end{align}
where in the last line we used that by the Cauchy theorem the integral does not depend on the choice of $\varepsilon$ and $R$ as long all zeros of $q(s)$ are inside $U_{\varepsilon,R}$.

\medskip

In the \textit{second step} we consider the operator
\begin{equation}\label{Eq_Rational_equivalence_13}
B:=\frac{1}{\pi}\int_{\partial(U_{\varepsilon,R}\cap\mathbb{C}_J)}Q_{c,s}^{-1}(T)ds_J\frac{1}{q(s)},
\end{equation}
for which we will prove that $\ran(B)\subseteq\dom(\overline{T}^j)$ and
\begin{equation}\label{Eq_Rational_equivalence_2}
\overline{T}^jB=\frac{1}{\pi}\int_{\partial(U_{\varepsilon,R}\cap\mathbb{C}_J)}Q_{c,s}^{-1}(T)ds_J\frac{s^j}{q(s)}+2\sum\limits_{k=0}^{j-1}T^k\overline{T}^{j-1-k}q[T]^{-1},\qquad j\in\{0,\dots,m\},
\end{equation}
where for $j=0$ the sum $\sum_{k=0}^{-1}$ is understood to be $0$. First of all, we note that $q[T]$ is bijective due to Lemma~\ref{lem_Polynomial_bijective}. The induction start $j=0$ is clear by the definition \eqref{Eq_Rational_equivalence_13} of the operator $B$. For the induction step $j-1\rightarrow j$ we use the identity
\begin{equation*}
Q_{c,s}^{-1}(T)s=\overline{T}Q_{c,s}^{-1}(T)+S_L^{-1}(s,T),
\end{equation*}
which follows immediately from the definition \eqref{Eq_S_resolvent} of the left $S$-resolvent, to get
\begin{align*}
\frac{1}{\pi}\int_{\partial(U_{\varepsilon,R}\cap\mathbb{C}_J)}Q_{c,s}^{-1}(T)ds_J\frac{s^j}{q(s)}&=\frac{1}{\pi}\int_{\partial(U_{\varepsilon,R}\cap\mathbb{C}_J)}\big(\overline{T}Q_{c,s}^{-1}(T)+S_L^{-1}(s,T)\big)ds_J\frac{s^{j-1}}{q(s)} \\
&=\overline{T}\Big(\overline{T}^{j-1}B-2\sum\limits_{k=0}^{j-2}T^k\overline{T}^{j-2-k}q[T]^{-1}\Big)-2T^{j-1}q[T]^{-1} \\
&=\overline{T}^jB-2\sum\limits_{k=0}^{j-1}T^k\overline{T}^{j-1-k}q[T]^{-1},
\end{align*}
where in the second line we used the induction assumption \eqref{Eq_Rational_equivalence_2} for $j-1$ as well as the identities \eqref{Eq_Polynomial_bijective_2} and \eqref{Eq_Polynomial_bijective_4}. Bringing the very right term to the left hand side finishes the induction step.

\medskip

In the \textit{third step} we show that
\begin{equation}\label{Eq_Rational_equivalence_3}
B=-Dq[T]q[T]^{-1}q[\overline{T}]^{-1}.
\end{equation}
Since $\overline{U_{\varepsilon,R}}\subseteq\rho_S(T)$, where $s\mapsto Q_{c,s}(T)$ is slice hyperholomorphic, we get
\begin{align*}
0&=\frac{1}{\pi}\int_{\partial(U_{\varepsilon,R}\cap\mathbb{C}_J)}Q_{c,s}^{-1}(T)ds_J1=\frac{1}{\pi}\int_{\partial(U_{\varepsilon,R}\cap\mathbb{C}_J)}Q_{c,s}^{-1}(T)ds_J\frac{q(s)}{q(s)} \\
&=\sum\limits_{j=0}^m\frac{b_j}{\pi}\int_{\partial(U_{\varepsilon,R}\cap\mathbb{C}_J)}Q_{c,s}^{-1}(T)ds_J\frac{s^j}{q(s)} \\
&=\sum\limits_{j=0}^mb_j\Big(\overline{T}^jB-2\sum\limits_{k=0}^{j-1}T^k\overline{T}^{j-1-k}q[T]^{-1}\Big)=q[\overline{T}]B+Dq[T]q[T]^{-1},
\end{align*}
where in the last line we used \eqref{Eq_Rational_equivalence_2} and the definition \eqref{Eq_Derivative_polynomial_operator} of $Dq[T]$. Using the bijectivity of $q[\overline{T}]$ from Lemma~\ref{lem_Polynomial_bijective}, this is exactly equation in \eqref{Eq_Rational_equivalence_3}.

\medskip

In the \textit{fourth step} we combine the results \eqref{Eq_Rational_equivalence_1}, \eqref{Eq_Rational_equivalence_2} and \eqref{Eq_Rational_equivalence_3} to get the stated equation
\begin{align*}
D\Big(\frac{p}{q}\Big)(T)&=\frac{1}{\pi}\int_{\partial(U_{\varepsilon,R}\cap\mathbb{C}_J)}Q_{c,s}^{-1}(T)ds_J\frac{p(s)}{q(s)}=\sum\limits_{i=0}^n\frac{a_i}{\pi}\int_{\partial(U_{\varepsilon,R}\cap\mathbb{C}_J)}Q_{c,s}^{-1}(T)ds_J\frac{s^i}{q(s)} \\
&=\sum\limits_{i=0}^na_i\Big(\overline{T}^iB-2\sum\limits_{k=0}^{i-1}T^k\overline{T}^{i-1-k}q[T]^{-1}\Big)=p[\overline{T}]B+Dp[T]q[T]^{-1} \\
&=\Big(Dp[T]q[T]-p[\overline{T}]Dq[T]\Big)q[T]^{-1}q[\overline{T}]^{-1}. \qedhere
\end{align*}
\end{proof}

\section{The harmonic $H^\infty$-functional calculus}\label{sec_Harmonic_functional_calculus_for_growing_functions}

In this section we want to enlarge the harmonic functional calculus of Definition~\ref{defi_Q_functional_calculus_decaying} to slice hyperholomorphic functions on a sector, which are polynomially growing at $0$ and at $\infty$. More precisely we consider the following function spaces.

\begin{defi}\label{defi_Space_of_increasing_functions}
For every $\theta\in(0,\pi)$ we define the function spaces

\begin{enumerate}
\item[i)] $\mathcal{F}_L(S_\theta):=\Set{f\in\mathcal{SH}_L(S_\theta) | \exists\alpha>0,\,C_\alpha\geq 0: |f(s)|\leq C_\alpha\big(|s|^\alpha+\frac{1}{|s|^\alpha}\big),\,s\in S_\theta}$
\item[ii)] $\mathcal{F}_R(S_\theta):=\Set{f\in\mathcal{SH}_R(S_\theta) | \exists\alpha>0,\,C_\alpha\geq 0: |f(s)|\leq C_\alpha\big(|s|^\alpha+\frac{1}{|s|^\alpha}\big),\,s\in S_\theta}$
\item[iii)] $\mathcal{F}(S_\theta):=\Set{f\in\mathcal{N}(S_\theta) | \exists\alpha>0,\,C_\alpha\geq 0: |f(s)|\leq C_\alpha\big(|s|^\alpha+\frac{1}{|s|^\alpha}\big),\,s\in S_\theta}$
\end{enumerate}
\end{defi}

First, we generalize the $S$-functional calculus of Definition~\ref{defi_S_functional_calculus_decaying} to these polynomially growing functions by some regularizing procedure. This is called the $H^\infty$-functional calculus and was for example already considered in \cite[Section 6.3]{CGK}.

\begin{defi}[The $H^\infty$-functional calculus]\label{defi_S_functional_calculus_growing}
Let $T\in\mathcal{KC}(V)$ be an injective operator of type $\omega$. Then for every $f\in\mathcal{F}_L(S_\theta)$, $\theta\in(\omega,\pi)$, we define the $H^\infty$-functional calculus
\begin{equation}\label{Eq_S_functional_calculus_growing}
f(T):=e(T)^{-1}(ef)(T),
\end{equation}
where $e\in\Psi(S_\theta)$ is such that $e(T)$ is injective and $ef\in\Psi_L(S_\theta)$. The operators $e(T)$ and $(ef)(T)$ are understood by the $S$-functional calculus \eqref{Eq_S_functional_calculus_decaying} for decaying functions.
\end{defi}

It is then proven in \cite[Theorem 6.3.4]{CGK} that such a regularizer function $d$ exists and that \eqref{Eq_S_functional_calculus_growing} is independent of $e$.

\medskip

Motivated by the product rule \eqref{Eq_Product_rule_left_decaying}, we will now also generalize the harmonic functional calculus from Definition~\ref{defi_Q_functional_calculus_decaying} to the functions from Definition \ref{defi_Space_of_increasing_functions}.

\begin{defi}[The harmonic $H^\infty$-functional calculus]\label{defi_Q_functional_calculus_growing}
Let $T\in\mathcal{KC}(V)$ be such that $T,\overline{T}$ are injective operators of type $\omega$. Then for every $f\in\mathcal{F}_L(S_\theta)$, $\theta\in(\omega,\pi)$ we define the \textit{harmonic $H^\infty$-functional calculus}
\begin{subequations}\label{Eq_Q_functional_calculus_growing}
\begin{align}
Df(T):=&e(\overline{T})^{-1}\big(D(ef)(T)-f(T)De(T)\big) \label{Eq_Q_functional_calculus_growing1} \\
=&e(T)^{-1}\big(D(ef)(T)-f(\overline{T})De(T)\big), \label{Eq_Q_functional_calculus_growing2}
\end{align}
\end{subequations}
where $e\in\Psi^Q(S_\theta)$, such that $e(T),e(\overline{T})$ are injective and $ef\in\Psi_L^Q(S_\theta)$. Here $D(ef)(T),De(T)$ are understood by the harmonic functional calculus \eqref{Eq_Q_functional_calculus_decaying} for decaying functions, $e(T),e(\overline{T})$ by the $S$-functional calculus \eqref{Eq_S_functional_calculus_decaying} for decaying functions and $f(T),f(\overline{T})$ by the $H^\infty$-functional calculus \eqref{Eq_S_functional_calculus_growing}.
\end{defi}

\begin{rem}
Note, that in the Definitions~\ref{defi_S_functional_calculus_growing}~\&~\ref{defi_Q_functional_calculus_growing} only left slice hyperholomorphic functions $f\in\mathcal{F}_L(S_\theta)$ are considered. The reason why the natural analogues
\begin{equation*}
f(T):=(fe)(T)e(T)^{-1}\quad\text{and}\quad(fD)(T):=\big(((fe)D)(T)-f(\overline{T})De(T)\big)e(T)^{-1},
\end{equation*}
of \eqref{Eq_S_functional_calculus_growing} and \eqref{Eq_Q_functional_calculus_growing2} for right slice hyperholomorphic functions $f\in\mathcal{F}_R(S_\theta)$ are missing, is that these terms would be defined only on $\dom(e(T)^{-1})$ and hence not independent of the choice of the regularizer; see also \cite[Remark 7.2.2]{FJBOOK}.
\end{rem}

\begin{thm}
Let $T\in\mathcal{KC}(V)$ be such that $T,\overline{T}$ are injective operators of type $\omega$. Then for every $f\in\mathcal{F}_L(S_\theta)$, $\theta\in(\omega,\pi)$ a regularizer function $e$ in the sense of Definition~\ref{defi_Q_functional_calculus_growing} exists, \eqref{Eq_Q_functional_calculus_growing1} and \eqref{Eq_Q_functional_calculus_growing2} equal and do not depend on the choice of the regularizer $e$. Moreover, for every $f,g\in\Psi_L(S_\theta)$ with $Df(s)=Dg(s)$, also the operators $Df(T)=Dg(T)$ coincide.
\end{thm}

\begin{proof}
For the independence of the regularizer $e$ let $e_1,e_2\in\Psi^Q(S_\theta)$ be two regularizers such that $e_1(T),e_1(\overline{T}),e_2(T),e_2(\overline{T})$ are injective and $e_1f,e_2f\in\Psi_L^Q(S_\theta)$. From the two versions of the product rule \eqref{Eq_Product_rule_left_decaying1} and \eqref{Eq_Product_rule_left_decaying2} applied to $D(e_1e_2f)(T)$, we get
\begin{equation}\label{Eq_Q_functional_calculus_decaying_14}
De_1(T)(e_2f)(T)+e_1(\overline{T})D(e_2f)(T)=De_2(T)(e_1f)(\overline{T})+e_2(T)D(e_1f)(T).
\end{equation}
Rearranging this equation and interchanging $De_1(T)$ and $(e_2f)(T)$ as well as $De_2(T)$ and $(e_1f)(\overline{T})$, which is allowed by Corollary~\ref{cor_Commutation_g_with_Df} and Lemma~\ref{lem_Properties_Q_functional_calculus_decaying}~iii), gives
\begin{equation}\label{Eq_Q_functional_calculus_decaying_12}
e_2(T)D(e_1f)(T)-(e_2f)(T)De_1(T)=e_1(\overline{T})D(e_2f)(T)-(e_1f)(\overline{T})De_2(T).
\end{equation}
Next, we note that $e_1(\overline{T})e_2(T)=e_2(T)e_1(\overline{T})$ commute due to \cite[Lemma 3.9]{CGdiffusion2018} and so do their inverses $e_1^{-1}(\overline{T})e_2^{-1}(T)=e_2^{-1}(T)e_1^{-1}(\overline{T})$. Multiplying these two inverses from the left to \eqref{Eq_Q_functional_calculus_decaying_12} and using the definition of the $H^\infty$-functional calculus \eqref{Eq_S_functional_calculus_growing} gives
\begin{equation}\label{Eq_Q_functional_calculus_decaying_11}
e_1(\overline{T})^{-1}\big(D(e_1f)(T)-f(T)De_1(T)\big)=e_2(T)^{-1}\big(D(e_2f)(T)-f(\overline{T})De_2(T)\big).
\end{equation}
This equation with the particular choice $e_1=e_2$ gives the identity
\begin{equation*}
e_2(\overline{T})^{-1}\big(D(e_2f)(T)-f(T)De_2(T)\big)=e_2(T)^{-1}\big(D(e_2f)(T)-f(\overline{T})De_2(T)\big).
\end{equation*}
This equation on the one hand shows the equivalence of the definitions \eqref{Eq_Q_functional_calculus_growing1} and \eqref{Eq_Q_functional_calculus_growing2}, and on the other hand plugging it into \eqref{Eq_Q_functional_calculus_decaying_11}, gives the independence of the regularizer
\begin{equation*}
e_1(\overline{T})^{-1}\big(D(e_1f)(T)-f(T)De_1(T)\big)=e_2(\overline{T})^{-1}\big(D(e_2f)(T)-f(T)De_2(T)\big).
\end{equation*}
It is left to show that a regularizer function $e\in\Psi^Q(S_\theta)$ with $e(\overline{T})$ injective and $ef\in\Psi_L^Q(S_\theta)$ exists. Since $f\in\mathcal{F}_L(S_\theta)$, it admits the estimate
\begin{equation*}
|f(s)|\leq C_\alpha\Big(|s|^\alpha+\frac{1}{|s|^\alpha}\Big),\qquad s\in S_\theta,
\end{equation*}
for some $\alpha>0$, $C_\alpha\geq 0$. For arbitrary $n\in\mathbb{N}$ with $n>1+\alpha$ we choose the regularizer
\begin{equation}\label{Eq_Q_functional_calculus_unbounded_3}
e(s):=\frac{s^n}{(1+s)^{2n-1}}.
\end{equation}
Then obviously $e\in\Psi^Q(S_\theta)$ and due to the asymptotics
\begin{equation*}
e(s)f(s)=\begin{cases} \mathcal{O}(|s|^{-n+1+\alpha}), & \text{as }s\rightarrow\infty, \\ \mathcal{O}(|s|^{n-\alpha}), & \text{as }s\rightarrow 0, \end{cases}
\end{equation*}
we also have $ef\in\Psi_L^Q(S_\theta)$. In order to show that $e(T)=T^n(1+T)^{-2n+1}$ is injective, we note that since $T$ is injective, the $n$-th power $T^n$ is injective as well. Since $(1+T)^{-2n+1}$ is bijective due to $-1\in\rho_S(T)$, we conclude that $e(T)$ is injective. The injectivity of $e(\overline{T})$ follows the same argument.

\medskip

In order to prove the independence of the kernel of $D$, let $f,g\in\Psi_L(S_\theta)$ with $Df(s)=Dg(s)$. Assume furthermore that $e$ is a regularizer for $f$ as well as for $g$. In the same way as in \eqref{Eq_Q_functional_calculus_decaying_16} one shows that $f=g+c$ only differ by a constant. Plugging this into \eqref{Eq_Q_functional_calculus_growing} gives
\begin{align*}
Df(T)&=e(\overline{T})^{-1}\big(D(ef)(T)-f(T)De(T)\big) \\
&=e(\overline{T})^{-1}\big(D(eg+ce)(T)-(g+c)(T)De(T)\big) \\
&=e(\overline{T})^{-1}\big(D(eg(T)+cDe(T)-g(T)De(T)-cDe(T)\big) \\
&=e(\overline{T})^{-1}\big(D(eg(T)-g(T)De(T)\big)=Dg(T),
\end{align*}
where in the third equation we used the linearity of the harmonic functional calculus in Lemma \ref{lem_Properties_Q_functional_calculus_decaying} ii) and the linearity $(g+c)(T)=g(T)+c$, which holds with equality if one of the operators is bounded.
\end{proof}

\begin{lem}\label{lem_Properties_Q_functional_calculus_growing}
Let $T\in\mathcal{KC}(V)$ with $T,\overline{T}$ being injective operators of type $\omega$. Then for every $f,g\in\mathcal{F}_L(S_\theta)$, $\theta\in(\omega,\pi)$, there holds

\begin{enumerate}
\item[i)] $Df(T)$ is a closed operator;
\item[ii)] $Df(T)+Dg(T)\subseteq D(f+g)(T)$;
\item[iii)] $Df(T)=Df(\overline{T})$.
\item[iv)] If $f$ is intrinsic, then $Df(T)\in\mathcal{KC}(V)$ with $\overline{Df(T)}=Df(T)$.
\end{enumerate}
\end{lem}

\begin{proof}
Let $e$ be a regularizer of $f$ according to Definition~\ref{defi_Q_functional_calculus_growing}. For the proof of i), we first use the definition \eqref{Eq_S_functional_calculus_growing} of the $H^\infty$-functional calculus to rewrite \eqref{Eq_Q_functional_calculus_growing1} as
\begin{equation}\label{Eq_Properties_Q_functional_calculus_growing_1}
Df(T)=\big(e(T)e(\overline{T})\big)^{-1}\big(e(T)D(ef)(T)-(ef)(T)De(T)\big).
\end{equation}
Since this is the combination of an unbounded closed with a bounded operator, also its decomposition $Df(T)$ turns out to be a closed operator again. Next, we note that due to Lemma~\ref{lem_Properties_S_functional_calculus_decaying}~iii), we know that $e(T)$ and $e(\overline{T})$ are bounded operators with commuting components and that
\begin{equation*}
e(T)e(\overline{T})=e(T)\overline{e(T)}=|e(T)|^2
\end{equation*}
only admits the component $(e(T)e(\overline{T}))_0$, while $(e(T)e(\overline{T}))_i=0$, $i\in\{1,2,3\}$, vanish. Hence also the inverse $(e(T)e(\overline{T}))^{-1}=(e(T)e(\overline{T}))_0^{-1}$ is an unbounded operator only consisting of the inverse of that one real component. Moreover, $e(T),D(ef)(T),(ef)(T),De(T)$ are operators with commuting components as well, see Lemma~\ref{lem_Properties_Q_functional_calculus_decaying} and Lemma~\ref{lem_Properties_S_functional_calculus_decaying}, and the respective components are combinations of powers of the components $T_0,T_1,T_2,T_3$, see the explicit representations \eqref{Eq_Properties_Q_functional_calculus_decaying_2} and \eqref{Eq_Properties_S_functional_calculus_decaying_2}. Hence also the decomposition \eqref{Eq_Properties_Q_functional_calculus_growing_1} is an operator with commuting component, which prove that $Df(T)\in\mathcal{KC}(V)$.

\medskip

For the proof of ii) we assume that $e$ is a regularizer of $f$ and of $g$. Then $e$ is a regularizer of the sum $f+g$ as well. Combining now the facts

\begin{enumerate}
\item[$\circ$] $AB+AC\subseteq A(B+C)$,\hspace{3.05cm} for unbounded operators $A,B,C$ in $V$,
\item[$\circ$] $D(ef)(T)+D(eg)(T)=D(ef+eg)(T)$,\quad by Lemma~\ref{lem_Properties_Q_functional_calculus_decaying},
\item[$\circ$] $f(T)+g(T)\subseteq(f+g)(T)$,\hspace{2.55cm} by \cite[Theorem 5.6]{ACQS2016},
\end{enumerate}

leads us to the inclusion
\begin{align*}
Df(T)+Dg(T)&=e(\overline{T})^{-1}\big(D(ef)(T)-f(T)De(T)\big)+e(\overline{T})^{-1}\big(D(eg)(T)-g(T)De(T)\big) \\
&\subseteq e(\overline{T})^{-1}\Big(D(ef)(T)-f(T)De(T)+D(eg)(T)-g(T)De(T)\Big) \\
&\subseteq e(\overline{T})^{-1}\Big(D(ef+eg)(T)-(f+g)(T)De(T)\Big)=D(f+g)(T).
\end{align*}
The proof of iii) follows from the two equivalent definitions \eqref{Eq_Q_functional_calculus_growing1} and \eqref{Eq_Q_functional_calculus_growing2} and Lemma~\ref{lem_Properties_Q_functional_calculus_decaying}~iii).

\medskip

Finally, for the statement iv) it first follows from Lemma \ref{lem_Properties_Q_functional_calculus_decaying} iii) and Lemma \ref{lem_Properties_S_functional_calculus_decaying} iii), that
\begin{align*}
\overline{e(T)D(ef)(T)-(ef)(T)De(T)}&=e(\overline{T})D(ef)(T)-(ef)(\overline{T})De(T) \\
&=e(T)D(ef)(T)-(ef)(T)De(T),
\end{align*}
where the second equation follows from the two versions of the product rule \eqref{Eq_Product_rule_left_decaying1} and \eqref{Eq_Product_rule_left_decaying2}. Hence the operator $e(T)D(ef)(T)-(ef)(T)De(T)$, and consequently also $Df(T)$ in the representation \eqref{Eq_Properties_Q_functional_calculus_growing_1}, is two-sided linear. This means that $Df(T)$ only consists of one component and automatically implies $Df(T)\in\mathcal{KC}(V)$ as well as $\overline{Df(T)}=Df(T)$.
\end{proof}

\begin{prop}\label{prop_Commutation_with_Df_growing}
Let $T\in\mathcal{KC}(V)$ with $T,\overline{T}$ being injective operators of type $\omega$, $B:V\rightarrow V$ an everywhere defined bounded operator and $f\in\mathcal{F}(S_\theta)$, for some $\theta\in(\omega,\pi)$. Then there holds
\begin{equation*}
\begin{array}{c} \forall j\in\{0,\dots,3\}: BT_j=T_jB\quad\text{on }\dom(T) \\ \Downarrow \\ BDf(T)\subseteq Df(T)B. \end{array}
\end{equation*}
\end{prop}

\begin{proof}
First of all, we note that since $BT_j=T_jB$ commutes with every component $T_j$ on $\dom(T)$, the operator $B$ commutes also with every component $e(T)_i$ of $e(T)$, see the representation \eqref{Eq_Properties_S_functional_calculus_decaying_2}. Hence $B$ also commutes with
\begin{equation*}
e(T)e(\overline{T})=e(T)\overline{e(T)}=e_1(T)_0^2+e(T)_1^2+e(T)_2^2+e(T)_3^2.
\end{equation*}
Consequently for the inverse we get the inclusion
\begin{align*}
B\big(e(T)e(\overline{T})\big)^{-1}&=\big(e(T)e(\overline{T})\big)^{-1}e(T)e(\overline{T})B\big(e(T)e(\overline{T})\big)^{-1} \\
&=\big(e(T)e(\overline{T})\big)^{-1}Be(T)e(\overline{T})\big(e(T)e(\overline{T})\big)^{-1}\subseteq\big(e(T)e(\overline{T})\big)^{-1}B.
\end{align*}
Since for the same reason, $B$ also commutes with the components of $e(T)$, $D(ef)(T)$, $(ef)(T)$ and $De(T)$, it also commutes with the components of the bracket term in \eqref{Eq_Properties_Q_functional_calculus_growing_1}. However, since $Df(T)=\overline{Df(T)}$ by Lemma~\ref{lem_Properties_Q_functional_calculus_growing}~iv), we know that $Df(T)$ and hence also $e(T)D(ef)(T)-(ef)(T)De(T)$ only consists of the real term, while all the components which are denoted on the imaginary units $e_1,e_2,e_3$ vanish identically. Hence
\begin{equation*}
B\big(e(T)D(ef)(T)-(ef)(T)De(T)\big)=\big(e(T)D(ef)(T)-(ef)(T)De(T)\big)B,\quad\text{on }V,
\end{equation*}
commutes with the whole bracket term in \eqref{Eq_Properties_Q_functional_calculus_growing_1} and hence also with
\begin{align*}
BDf(T)&=B\big(e(T)e(\overline{T})\big)^{-1}\big(e(T)D(ef)(T)-(ef)(T)De(T)\big) \\
&\subseteq\big(e(T)e(\overline{T})\big)^{-1}B\big(e(T)D(ef)(T)-(ef)(T)De(T)\big) \\
&=\big(e(T)e(\overline{T})\big)^{-1}\big(e(T)D(ef)(T)-(ef)(T)De(T)\big)B=Df(T)B. \qedhere
\end{align*}
\end{proof}

\begin{cor}\label{cor_Commutation_g_with_Df_growing}
Let $T\in\mathcal{KC}(V)$ with $T,\overline{T}$ being injective operators of type $\omega$. Then

\medskip

\begin{enumerate}
\item[i)] $Df(T)g(T)\subseteq g(T)Df(T)$\quad for $f\in\Psi^Q(S_\theta),\,g\in\mathcal{F}_L(S_\theta)$
\item[ii)] $g(T)Df(T)\subseteq Df(T)g(T)$\quad for $f\in\mathcal{F}^Q(S_\theta),\,g\in\Psi_L^Q(S_\theta)$ or $g\in\Psi_R^Q(S_\theta)$
\end{enumerate}
\end{cor}

\begin{proof}
For the proof of i) let $e$ be a regularizer of $g$ according to Definition~\ref{defi_S_functional_calculus_growing}. Then
\begin{equation}\label{Eq_Commutation_g_with_Df_increasing_1}
Df(T)g(T)=Df(T)e(T)^{-1}(eg)(T)\subseteq e(T)^{-1}(eg)(T)Df(T)=g(T)Df(T),
\end{equation}
where we commuted
\begin{equation}\label{Eq_Commutation_g_with_Df_increasing_2}
Df(T)e(T)^{-1}=e(T)^{-1}e(T)Df(T)e(T)^{-1}=e(T)^{-1}Df(T)e(T)e(T)^{-1}\subseteq e(T)^{-1}Df(T),
\end{equation}
and $Df(T)(eg)(T)=(eg)(T)Df(T)$, which is allowed by Corollary \ref{cor_Commutation_g_with_Df}.

\medskip

The inclusion in ii) follows from Proposition~\ref{prop_Commutation_with_Df_growing}, since $g(T)T_j=T_jg(T)$ obviously commute on $\dom(T)$ by the integral definition of the $S$-functional calculus \eqref{Eq_S_functional_calculus_decaying}.
\end{proof}

\begin{thm}\label{thm_Product_rule_growing}
Let $T\in\mathcal{KC}(V)$ with $T,\overline{T}$ being injective operators of type $\omega$. Then for every $f\in\mathcal{F}(S_\theta)$, $g\in\mathcal{F}_L(S_\theta)$, $\theta\in(\omega,\pi)$, there holds the product rules
\begin{subequations}
\begin{align}
D(fg)(T)&\supseteq Df(T)g(T)+f(\overline{T})Dg(T),\quad\text{and} \label{Eq_Product_rule_growing1} \\
D(fg)(T)&\supseteq Df(T)g(\overline{T})+f(T)Dg(T). \label{Eq_Product_rule_growing2}
\end{align}
\end{subequations}
\end{thm}

\begin{proof}
First of all, the second rule \eqref{Eq_Product_rule_growing2} follows immediately when we replace $T$ by $\overline{T}$ in \eqref{Eq_Product_rule_growing1} and use Lemma~\ref{lem_Properties_Q_functional_calculus_growing}~iii). For the prove of \eqref{Eq_Product_rule_growing1} let $e_1$ be a regularizer of $f$ and $e_2$ a regularizer of $g$ according to Definition~\ref{defi_Q_functional_calculus_growing}. Using the product rule $(fg)(T)\supseteq f(T)g(T)$ of the $S$-functional calculus \cite[Theorem 5.6]{ACQS2016} and the commutation of $g(T)$ and $D(e_1e_2)(T)$ as in \eqref{Eq_Commutation_g_with_Df_increasing_1}, gives
\begin{equation}\label{Eq_Product_rule_growing_1}
(fg)(T)D(e_1e_2)(T)\supseteq f(T)g(T)D(e_1e_2)(T)\supseteq f(T)D(e_1e_2)(T)g(T).
\end{equation}
Since $e_1e_2$ is a regularizer of the product $fg$, we can use the identity \eqref{Eq_Product_rule_growing_1} and the product rule \eqref{Eq_Product_rule_left_decaying1}, to get
\begin{align*}
D(fg)(T)&=(e_1e_2)(\overline{T})^{-1}\big(D(e_1e_2fg)(T)-(fg)(T)D(e_1e_2)(T)\big) \\
&\supseteq(e_1e_2)(\overline{T})^{-1}\big(D(e_1fe_2g)(T)-f(T)D(e_1e_2)(T)g(T)\big) \\
&\supseteq\big(e_1(\overline{T})e_2(\overline{T})\big)^{-1}\big(D(e_1f)(T)(e_2g)(T)+(e_1f)(\overline{T})D(e_2g)(T) \notag \\
&\hspace{3.2cm}-f(T)De_2(T)e_1(T)g(T)-f(T)e_2(\overline{T})De_1(T)g(T).
\end{align*}
Using now the commutation rules

\begin{enumerate}
\item[$\circ$] $D(e_1f)(T)(e_2g)(T)\supseteq D(e_1f)(T)e_2(T)e_2(T)^{-1}(e_2g)(T)=e_2(T)D(e_1f)(T)g(T)$,
\item[$\circ$] $De_2(T)e_1(T)=e_1(T)De_2(T)$,\hspace{1cm} from Corollary~\ref{cor_Commutation_g_with_Df},
\item[$\circ$] $f(T)e_1(T)=e_1(T)^{-1}(e_1f)(T)e_1(T)=(e_1f)(T)$,
\item[$\circ$] $f(T)e_2(\overline{T})\supseteq e_2(\overline{T})f(T)$,\hspace{1.95cm} as in \eqref{Eq_Commutation_g_with_Df_increasing_1},
\end{enumerate}

we can rearrange the terms and obtain the product rule
\begin{align*}
D(fg)(T)&\supseteq e_1(\overline{T})^{-1}e_2(\overline{T})^{-1}\big(e_2(T)D(e_1f)(T)g(T)+(e_1f)(\overline{T})D(e_2g)(T) \notag \\
&\hspace{3.2cm}-(e_1f)(T)De_2(T)g(T)-e_2(\overline{T})f(T)De_1(T)g(T)\big) \\
&=e_1(\overline{T})^{-1}e_2(\overline{T})^{-1}\big(e_2(\overline{T})D(e_1f)(T)g(T)+(e_1f)(\overline{T})D(e_2g)(T) \notag \\
&\hspace{3.2cm}-(e_1f)(\overline{T})De_2(T)g(T)-e_2(\overline{T})f(T)De_1(T)g(T)\big) \\
&\supseteq e_1(\overline{T})^{-1}\big(D(e_1f)(T)-f(T)De_1(T)\big)g(T) \\
&\quad+e_1(\overline{T})^{-1}(e_1f)(\overline{T})e_2(\overline{T})^{-1}\big(D(e_2g)(T)-De_2(T)g(T)\big) \\
&=Df(T)g(T)+f(\overline{T})Dg(T),
\end{align*}
where in the second line we used \eqref{Eq_Q_functional_calculus_decaying_14} and in the third one we commuted $e_2(\overline{T})^{-1}(e_1f)(\overline{T})\supseteq(e_1f)(\overline{T})e_2(\overline{T})^{-1}$ by a similar argument as in \eqref{Eq_Commutation_g_with_Df_increasing_2}.
\end{proof}

As the final result of this paper, we investigate the action of the harmonic functional calculus on intrinsic rational functions. The following Proposition~\ref{prop_Rational_equivalence_growing} contains basically the same statement as Proposition~\ref{prop_Rational_equivalence_decaying}, but without the assumptions i), ii) on the decay on the polynomials. As a consequence we are obliged to use the harmonic $H^\infty$-functional calculus \eqref{Eq_Q_functional_calculus_growing} instead of the calculus \eqref{Eq_Q_functional_calculus_decaying} for decaying functions.

\begin{prop}\label{prop_Rational_equivalence_growing}
Let $T\in\mathcal{KC}(V)$ with $T,\overline{T}$ being injective operators of type $\omega$. Then for every intrinsic polynomials $p,q$, with $q$ not having any zeros in the closed sector $\overline{S_\omega}$, we know that $q[T]$, $q[\overline{T}]$ are bijective and
\begin{align*}
D\Big(\frac{p}{q}\Big)(T)&=\big(Dp[T]q[T]-p[T]Dq[T]\big)q[T]^{-1}q[\overline{T}]^{-1} \\
&=\big(Dp[T]q[T]-p[T]Dq[T]\big)q[T]^{-1}q[\overline{T}]^{-1}.
\end{align*}
Here, the left hand side of this equation is understood as the harmonic $H^\infty$-functional calculus \eqref{Eq_Q_functional_calculus_growing} and the operators on the right via the polynomial functional calculi \eqref{Eq_Polynomial_functional_calculus} and \eqref{Eq_Derivative_polynomial_operator}.
\end{prop}

\begin{proof}
In the \textit{first step}, let $p$ and $q$ be any intrinsic polynomials. Then, in the sense of the polynomial functional calculi \eqref{Eq_Polynomial_functional_calculus} and \eqref{Eq_Derivative_polynomial_operator} we obtain the product rule
\begin{align}
Dp[T]q[T]+p[\overline{T}]Dq[T]&=-2\sum\limits_{i=0}^np_i\sum\limits_{k=0}^{i-1}T^k\overline{T}^{i-1-k}\sum\limits_{j=0}^mq_jT^j-2\sum\limits_{i=0}^np_i\overline{T}^i\sum\limits_{j=0}^mq_j\sum\limits_{k=0}^{j-1}T^k\overline{T}^{j-1-k} \notag \\
&=-2\sum\limits_{i=0}^n\sum\limits_{j=0}^mp_iq_j\bigg(\sum\limits_{k=0}^{i-1}T^{j+k}\overline{T}^{i-1-k}+\sum\limits_{k=0}^{j-1}T^k\overline{T}^{i+k-1-k}\bigg) \notag \\
&=-2\sum\limits_{i=0}^n\sum\limits_{j=0}^mp_iq_j\sum\limits_{k=0}^{i+j-1}T^k\overline{T}^{i+j-1-k}=D(pq)[T]. \label{Eq_Rational_Q_equivalence_growing_2}
\end{align}
In the \textit{second step} we choose an intrinsic rational regularizer $e(s)=\frac{r(s)}{t(s)}$, for example the one in \eqref{Eq_Q_functional_calculus_unbounded_3} with large enough $n\in\mathbb{N}$. Then it follows from Proposition~\ref{prop_Rational_equivalence_decaying} and \eqref{Eq_Rational_Q_equivalence_growing_2}, that
\begin{align}
D\Big(\frac{rp}{tq}\Big)(T)&=\Big(D(rp)[T](tq)[T]-(rp)[T]D(tq)[T]\Big)(tq)[T]^{-1}(tq)[\overline{T}]^{-1} \notag \\
&=\Big(Dr[T]p[T]t[T]q[T]+r[\overline{T}]Dp[T]t[T]q[T] \notag \\
&\hspace{1cm}-r[T]p[T]Dt[T]q[\overline{T}]-r[T]p[T]t[T]Dq[T]\Big)(tq)[T]^{-1}(tq)[\overline{T}]^{-1}, \label{Eq_Rational_Q_equivalence_growing_3}
\end{align}
as well as
\begin{align}
\Big(\frac{p}{q}\Big)(T)D\Big(\frac{r}{t}\Big)(T)&=p[T]q[T]^{-1}\big(Dr[T]t[T]-r[T]Dt[T]\big)t[T]^{-1}t[\overline{T}]^{-1} \notag \\
&=\big(Dr[T]p[T]t[T]q[\overline{T}]-r[T]p[T]Dt[T]q[\overline{T}]\big)(tq)[T]^{-1}(tq)[\overline{T}]^{-1}, \label{Eq_Rational_Q_equivalence_growing_4}
\end{align}
where we used that $(\frac{p}{q})(T)=p[T]q[T]^{-1}$ by \cite[Lemma 7.2.9]{FJBOOK}. Plugging the two representations \eqref{Eq_Rational_Q_equivalence_growing_3} and \eqref{Eq_Rational_Q_equivalence_growing_4} into the definition \eqref{Eq_Q_functional_calculus_growing}, we get the stated representation
\begin{align*}
D\Big(\frac{p}{q}\Big)(T)&=\Big(\frac{r}{t}\Big)(\overline{T})^{-1}\Big(D\Big(\frac{rp}{tq}\Big)(T)-\Big(\frac{p}{q}\Big)(T)D\Big(\frac{r}{t}\Big)(T)\Big) \\
&=\big(r[\overline{T}]t[\overline{T}]^{-1}\big)^{-1}\Big(Dr[T]p[T]t[T]q[T]+r[\overline{T}]Dp[T]t[T]q[T] \\
&\hspace{3.5cm}-r[T]p[T]t[T]Dq[T]-Dr[T]p[T]t[T]q[\overline{T}]\Big)(tq)[T]^{-1}(tq)[\overline{T}]^{-1} \\
&=r[\overline{T}]^{-1}\Big(Dr[T]p[T]q[T]+r[\overline{T}]Dp[T]q[T] \\
&\hspace{3.5cm}-r[T]p[T]Dq[T]-Dr[T]p[T]q[\overline{T}]\Big)q[T]^{-1}q[\overline{T}]^{-1} \\
&=\big(Dp[T]q[T]-Dq[T]p[T]\big)q[T]^{-1}q[\overline{T}]^{-1},
\end{align*}
where in the last equation we use the identity
\begin{equation*}
Dr[T]q[\overline{T}]+r(T)Dq[T]=Dr[T]q[T]+r[\overline{T}]Dq[T],
\end{equation*}
which can be checked by plugging in the definitions \eqref{Eq_Polynomial_functional_calculus} and \eqref{Eq_Derivative_polynomial_operator}.
\end{proof}

\end{document}